\documentclass[12pt]{amsart}
\usepackage{amssymb,amsmath}
\usepackage{enumerate}
\usepackage{tikz}
\usepackage{xcolor}
\usepackage{mathrsfs}
\usepackage[pagewise]{lineno}
\usepackage{amsfonts}
\usepackage{mathtools}
\allowdisplaybreaks
\date{}
 \usepackage[bookmarksnumbered, colorlinks, plainpages]{hyperref}
  \hypersetup{colorlinks=true,linkcolor=blue, anchorcolor=green, citecolor=cyan, urlcolor=blue, filecolor=magenta, pdftoolbar=true}

\newcommand{\beqa}{\begin{eqnarray*}}
\newcommand{\eeqa}{\end{eqnarray*}}
\newcommand{\beqn}{\begin{eqnarray}}
\newcommand{\eeqn}{\end{eqnarray}}

\newcommand{\N}{\mathbb N}

\newcommand{\A}{\mathcal{A}}

\newcommand{\B}{\mathcal{B}}

\newcommand{\M}{\mathcal{M}}
\newcounter{cnt1}
\newcounter{cnt2}
\newcounter{cnt3}
\newcommand{\blr}{\begin{list}{$($\roman{cnt1}$)$}
 {\usecounter{cnt1} \setlength{\topsep}{0pt}
 \setlength{\itemsep}{0pt}}}
\newcommand{\bla}{\begin{list}{$($\alph{cnt2}$)$}
 {\usecounter{cnt2} \setlength{\topsep}{0pt}
 \setlength{\itemsep}{0pt}}}
\newcommand{\bln}{\begin{list}{$($\arabic{cnt3}$)$}
 {\usecounter{cnt3} \setlength{\topsep}{0pt}
 \setlength{\itemsep}{0pt}}}
\newcommand{\el}{\end{list}}
\newtheorem{thm}{Theorem}[section]

\newtheorem{lemma}[thm]{Lemma}
\newtheorem{cor}[thm]{Corollary}
\newtheorem{ex}[thm]{Example}

\newtheorem{Def}[thm]{Definition}
\newtheorem{proposition}[thm]{Proposition}
\newtheorem{rem}[thm]{Remark}
\newcommand{\Rem}{\begin{rem} \rm}
\newcommand{\bdfn}{\begin{Def} \rm}
\newcommand{\edfn}{\end{Def}}

\newcommand{\ba}{\begin{array}}
\newcommand{\ea}{\end{array}}

\newcommand{\ben}{\begin{enumerate}}
\newcommand{\een}{\end{enumerate}}
\begin{document}

\sloppy

\title[\tiny{Russo-Dye Theorem, Stinespring Representation, and Radon Nikod\'ym Dervative}]{ Russo-Dye type Theorem, Stinespring representation, 
and Radon Nikod\'ym dervative for invariant block multilinear completely positive maps
}

\author[Ghatak]{Anindya Ghatak}
\address[Anindya Ghatak]{Stat--Math Division\\
Indian Statistical Institute\\8th Mile, Mysore Road\\ Bangalore 560059\\India, E-mail~: {\tt anindya.ghatak123@gmail.com}}

\author[Sensarma]{Aryaman Sensarma}
\address[Aryaman Sensarma]{Stat--Math Division\\
Indian Statistical Institute\\8th Mile, Mysore Road\\ Bangalore 560059\\India, E-mail~: {\tt aryamansensarma@gmail.com}}


\begin{abstract}
 In this article, we investigate certain basic properties of invariant multilinear CP maps. For instance, we prove Russo-Dye type theorem for invariant multilinear positive maps on both commutative $C^*$-algebras and finite-dimensional $C^*$-algebras.   
 We show that every invariant multilinear CP map is automatically symmetric and completely bounded. Possibly these results are unknown in the literature (see \cite{Heo 00,Heo,HJ 2019}). 
Motivated from quantum algorithm simulation \cite{BSD} we introduce  multilinear version of invariant block CP map $ \varphi=[\varphi_{ij}] : M_{n}(\A)^k \to M_n(\mathcal{B({H})}).$  Then we derive that each $\varphi_{ij}$ can be dilated to a common commutative tuple of
 $*$-homomorphisms. As a natural appeal, the suitable notion of minimality 
 has been identified within this framework. A special case of our result 
 recovers a finer version of J. Heo's  Stinespring type dilation theorem
 of  \cite{Heo}, and A. Kaplan's Stinespring type dilation theorem \cite{AK89}. As an application, we show Russo-Dye type theorem for invariant
 multilinear completely positive maps. Finally, using minimal 
 Stinespring dilation we obtain Radon Nikod\'ym theorem in this setup. Our result includes as a special case the Radon-Nikodým theorem of J. Heo from \cite{Heo} and the Radon-Nikodým theorem of M. Joița from \cite{Joita2}.
\end{abstract}


\keywords{Completely positive map, Completely bounded map, Multilinear map, Stinespring dilation, Russo-Dye theorem, Radon Nikody\'m theorem \\ \hspace*{\fill} {\bf Version: \today}}

\maketitle

\tableofcontents
\section{Introduction}
 In the late 1960s, W. B. Arveson used CP maps to define non-commutative measures in a seminal paper \cite{W. Arveson69}. Subsequently, he obtained the Radon-Nikod\'ym theorem of CP maps using the minimality condition of the Stinespring dilation theorem. 
   Over the decades, the theory of CP maps has been widely used as a tool in the study of $C^*$-algebras, and Operator Theory \cite{Joita,Joita1}. Moreover, CP maps has been used in great interest to study Markov maps of quantum probability, quantum channel in 
 quantum information theory, and in the classification problem of $C^*$-algebras (see \cite{Paulsen} and references therein).  
 
 In 1987, E. Christensen and A. M. Sinclair studied symmetric multilinear completely positive maps and completely bounded maps $\varphi:\mathcal{A}^{k}\to \mathcal{B(H)}$ \cite[p. 153, p. 155]{Christensen 87}. They obtained a technical Stinespring type theorem
  \cite[Theorem 4.1]{Christensen 87} 
 using Wittstock Hahn-Banach theorem (see \cite{Wittstock 83}). In 1998, A. M. Sinclair and R. R. Smith characterized that a von Neumann algebra is \emph{injective} if and only if every bilinear completely bounded map can be factorized in
  terms of bilinear CP maps \cite[Theorem 6.1]{SS98}. For more details on completely bounded multilinear maps and their intrinsic connection with the Haagerup tensor norm, we refer to \cite[Chapter 17]{Paulsen}.
  
  In 1990, J. P. Antoine and A. Inoue studied representations of invariant sesquilinear forms on partial $*$-algebras as these representations are arising from statistical mechanics and quantum field theory \cite{AI}.  
  Motivated by the work of J. P. Antoine and A. Inoue \cite{AI},  J. Heo studied multilinear CP maps using invariant property \cite{Heo 00, Heo, Heo 04}. Imposing invariant property, 
the author obtained an excellent form of the Stinespring dilation theorem for invariant multilinear completely positive maps which are completely bounded and symmetric as well (cf. \cite{Heo 00, Heo}).

A block map $\varphi= [\varphi_{ij}]: M_{n}(\mathcal{A})\to M_{n}(\mathcal{B})$  is defined by $$[\varphi_{ij}]([a_{ij}])= [\varphi_{ij}(a_{ij})]
\text{ for all } [a_{ij}]\in M_{n}(\mathcal{A}).$$
Given a fixed $[\lambda_{ij}]$  positive definite matrix, we can associate a block map $[\varphi_{ij}]:M_n(\mathbb{C})\longrightarrow M_n(\mathbb{C})$ defined by $$[\varphi_{ij}][a_{ij}]=[\lambda_{ij}a_{ij}] \text{ for all } [a_{ij}] \in M_n(\mathbb{C}),$$ is 
 a positive map. Hence block states can be thought of as a generalization of Schur products.
  Block CP maps are of independent interest and these provide various applications (see \cite{BV20, Paulsen} and references therein).  In 1999, J. Heo extensively studied block CP maps and derived the Stinespring type theorem using Hilbert $C^*$-module technique. 
  
The Aaronson-Ambanis conjecture is a significant proposition in the analysis of boolean functions. It states that for a degree $d$-multilinear polynomial $f:\{\pm 1\}^n \longrightarrow [0,1]$, the variable with the highest influence in $f$ must have an influence of at least $\text{poly}(\text{var}[f], \frac{1}{d})$. This conjecture raises a structural question about bounded polynomials on the hypercube and remains an open problem of great interest. In a related work, referenced in \cite{BSD}, the authors provide a proof for a different result. They show that in any completely bounded degree $d$ block-multilinear form with constant variables, there always exists a variable with an influence of at least $\frac{1}{\text{poly}(d)}$. This finding contributes to the exploration of bounded polynomials on the hypercube and provides insights into the conjecture. Furthermore, Montanaro has made progress on the Aaronson-Ambanis conjecture by proving a special case in which the block-multilinear forms satisfy the condition of having all coefficients with the same magnitude. This result, described in \cite{Montanaro}, narrows down the scope of the conjecture and adds to our understanding of the underlying principles. For more comprehensive information and specific details on these findings, we recommend referring to the cited sources \cite{AA,Mid05}. Block multilinear forms hold immense significance across multiple domains, with applications ranging from quantum algorithm simulation \cite{BSD} to other relevant fields.
 
Drawing inspiration from quantum algorithm simulation \cite{BSD} (thanks to Prof. Mathew Graydon for sharing such ideas), A. Kaplan's block states \cite{AK89} and J. Heo's work on invariant multilinear CP maps \cite{Heo 00}, we present a novel extension known as block invariant multilinear CP maps. For a detailed and comprehensive exploration of this topic, refer to Section \ref{Multi CP Maps}. 
 
 Our proposed model unifies the works of A. Kaplan \cite{AK89}, and J. Heo \cite{Heo 00,Heo} (see Section \ref{conclud remark} for more details). In addition, we provide several concrete examples to illustrate the concepts discussed in this topic. These examples serve to demonstrate the usefulness and applicability of the methods and ideas presented, and provide a deeper understanding of the theory.

Throughout the article, we assume that $\mathcal{A}, \mathcal{B}$ are unital $C^*$-algebras, and $\mathcal{A}_{sa}$ is the set of all self-adjoint elements of  $\mathcal{A}.$ 
 Let  $\varphi: \mathcal{A}\to \mathcal{B}$ be a linear map between two $C^*$-algebras. The \emph{adjoint} of a linear map $\varphi: \mathcal{A}\to \mathcal{B}$ is defined by
 $$
 \varphi^{*}(a) := \varphi(a^*)^* \text{ for all $a\in \mathcal{A}$}.
 $$
A linear map $\varphi$ is said to be \emph{symmetric} if $\varphi^*=\varphi$ and $\varphi$ is said to be \emph{positive} if $\varphi(a) \geq 0$ whenever $a\geq 0.$  It is  worth mentioning that every positive
map   is automatically symmetric and bounded. In 1966, Russo-Dye proved an important result of positive maps.
 \begin{thm} [Russo-Dye \cite{Paulsen,RD66}]\label{Russo-Dye} Let $\varphi: \mathcal{A}\to  \mathcal{B}$ be a positive map. Then $\Vert \varphi\Vert= \Vert \varphi(1)\Vert$.
 \end{thm}
 
Let $\varphi: \mathcal{A}\to \mathcal{B}$ be a  linear map, then $t$-\emph{amplification} map $\varphi_{t}: M_{t}( \mathcal{A})\to M_{t}(\mathcal{B})$ is defined by
 $\varphi_{t}[a_{ij}]= [\varphi(a_{ij})]$ for all $[a_{ij}]\in M_{t}(\mathcal{A}).$ A map $\varphi: \mathcal{A}\to \mathcal{B}$ is said to be \emph{completely positive} (in short: CP) if $\varphi_{t}$ is positive for each $t\in \N.$
  The map $\varphi$ is said to be \emph{completely bounded} if  $\sup\{ \Vert \varphi_{t} \Vert: t\in \mathbb{N}\}<\infty.$ Moreover, completely bounded norm of $\varphi$ is denoted and defined by
 $$\Vert \varphi \Vert_{cb}:= \sup\{ \Vert \varphi_{t} \Vert: t\in \mathbb{N}\}.$$  
For fixed $k \geq 1$ and $m=[\frac{k+1}{2}],$ set
 \begin{align}
\A^k&:=\underbrace{ \A\times \A \times \cdots \times\A}_{k \text{-times}},
&M_{t}\big(\A\big)^k&:=\underbrace{M_{t}\big(\A \big)\times M_{t}\big(\A \big)\times \cdots \times M_{t}\big(\A \big)}_{k \text{-times}}  \text{~for every~} t\in \N.
\end{align}
Let $\varphi:  \A^k \to \mathcal{B}$ be a multilinear map. Then \emph{$t$-amplification} map  of $\varphi$ is a multilinear map
 $\varphi_t: M_{t}\big(\A\big)^k\to M_t(\mathcal{B})$ defined by 
 \begin{align}
 \varphi_t([a_{{1}_{ij}}],[a_{{2}_{ij}}],\ldots ,[a_{{k}_{ij}}])=\left[\sum_{r_1,r_2,\ldots,r_{k-1}=1}^t \varphi\left (a_{{1}_{ir_1}}, a_{{2}_{r_1r_2}}, \ldots, a_{{k}_{r_{k-1}j}}\right)
 \right].\end{align}
 Moreover, the completely bounded norm of $\varphi$ is defined by (see  \cite{Christensen 87}) 
\begin{align*}
&\Vert \varphi \Vert_{cb}:= \sup\{ \Vert \varphi_{t} \Vert: t\in \mathbb{N}\},~\mbox{where},\\
& \Vert \varphi_{t} \Vert=\sup\{\Vert \varphi_t([a_{{1}_{ij}}],[a_{{2}_{ij}}],\ldots ,[a_{{k}_{ij}}])\Vert : ~ \Vert [a_{{l}_{ij}}] \Vert \leq 1~~\mbox{for all}~ 1 \leq l \leq k\}. 
\end{align*}
\begin{Def} [E. Christensen, A. M. Sinclair \cite{Christensen 87}] 
A  multilinear map $\varphi: \A ^k \to \mathcal{B}$ is said to be
\begin{enumerate}
\item \emph{positive} if $\varphi( a_{{1}}, a_{{2}},\ldots , a_{k}) \geq 0,$ whenever $(a_{{1}}, a_{{2}},\ldots , a_{{k}})=(a_{k}^*, a_{k-1}^*,\ldots, a_{1}^*)$ with an additional assumption $a_{m}\geq 0$ when $k$ is odd;
\item \emph{completely positive} if $\varphi_{t}$ is positive for each $t\in \N.$ 
 \end{enumerate}
\end{Def}
\noindent  The \emph{adjoint} of a multilinear map $\varphi : \A^k\to \mathcal{B}$ is denoted and defined by
 $$
 \varphi^*(a_{1}, a_{2}, \cdots,a_{k}):=\varphi(a_{k}^*, \cdots, a_{2}^*,a_{1}^*)^* \text{  for all  }a_{1}, a_{2}, \cdots, a_{k}\in \A.
 $$
A multilinear map $\varphi : \A^k\to \mathcal{B}$ is called \emph{symmetric} if $\varphi^*=\varphi$ \cite{Christensen 87}.
 \begin{Def}[J. Heo \cite{Heo}]
A multilinear map $\varphi: \mathcal{A}^{k}\to \mathcal{B}$ is said to be invariant if it satisfies the following:
\begin{enumerate}
\item[(i.)]  
$ \varphi \left( a_{1}c_{1},\hdots, a_{m-1} c_{m-1}, a_{m}, a_{m+1}, \hdots, a_{2m-1}\right)
    =\varphi \left( a_{1}, \hdots, a_{m-1}, a_m , c_{m-1} a_{m+1}, \hdots, c_{1} a_{2m-1} \right)$
when $k=2m-1$ and
for all $a_{1}, \hdots,a_{2m-1}, c_{1}, \hdots, c_{m-1} \in \mathcal{A}.$
\item[(ii.)]  
$    \varphi \left(a_{1} c_{1},\hdots, a_{m} c_{m}, a_{m+1}, \hdots, a_{2m}\right)
    =\varphi\left( a_{1},\hdots, a_{m-1}, a_{m} , c_{m}a_{m+1}, \hdots,c_{1} a_{2m} \right)$
 when $k=2m$ and for all $a_{1}, a_{2},\cdots ,a_{2m}, c_{1}, c_{2},\cdots, c_{m} \in \mathcal{A}.$  
\end{enumerate}
 \end{Def}
 Now, we fix some notations which are useful in the later sections.
 \begin{enumerate}
   \item $IP(\mathcal{A}^k, \mathcal{B})$:  The set of all \emph{invariant  positive $k$-linear} maps from $\mathcal{A}^k$ to $\mathcal{B}.$
 \item $ICP(\mathcal{A}^k,\mathcal{B})$:  The set of all   \emph{invariant completely positive $k$-linear} maps from $\mathcal{A}^k$ to $\mathcal{B}.$
 \end{enumerate}
We now briefly describe the plan of the article. In the Section \ref{RDT}, we study basic properties of multilinear CP maps. If an invariant multilinear positive map $\varphi: C(X)^{k}\to \mathcal{B(H)}$ is positive, then $\varphi$ 
attains its norm at $(1,\hdots, 1)$ (see Theorem \ref{Thm: norm attainment theorem}). 
  In the proof, we have extensively used the method of partition of unity.
  This can be thought of as a commutative multilinear analogue of Russo-Dye Theorem \cite{RD66}.  We also study Russo-Dye type theorem for invariant multilinear positive maps on finite-dimensional 
  $C^*$-algebras (see Theorem~\ref{RDND}). Russo-Dye theorem has a vast literature in various contexts in operator algebras (see \cite{AK96,Davidson}). Hence, we believe that these results are of
independent interest and will be useful in future investigations. We show that every multilinear invariant positive map is symmetric (see Proposition \ref{Prop: positivity implies symmetric}) and completely bounded (see Proposition \ref{invariant CP implies CB}).  
  In Example \ref{ex1}, we obtain a counter example that an invariant positive multilinear map $\varphi$ on a commutative domain $(\mathbb{C}^{2})^3$ is not completely positive. This result differs from a well-known theorem of Stinespring that every positive map on a commutative $C^*$-algebra is CP \cite[Theorem 3.11]{Paulsen}.

           In the Section \ref{Multi CP Maps}, we introduce block multilinear CP map. In Theorem \ref{Theorem: BlockStinespring}, we obtain Stinespring type theorem for block multilinear invariant CP maps.
            In this framework, we also address the suitable minimality condition.  In Theorem \ref{Theorem: Minimal BlockStinespring}, we prove that such minimality condition is unique up to unitary equivalence.
             In the Section \ref{Rdaon},  we obtain Radon Nikod\'ym theorem in this framework (see  Theorem \ref{thm: Radon Nikodym derivative}).
     Finally, as an application of multilinear Stinespring type theorem, a version of Russo-Dye type theorem has been derived for invariant multilinear CP maps (see Corollary \ref{Cor: Multilinear Russo Dye}).         
  \section{Russo-Dye type theorem and Properties of multilinear positive maps} \label{RDT}
 In the next result, we prove  multilinear analogue of Russo-Dye type Theorem (see Theorem \ref{Russo-Dye}). More precisely, this result ensures the norm attainment of  multilinear invariant positive maps on 
  commutative domain at $(1, 1, \cdots ,1)$ point. The idea of the proof is based on using the delicate structure of partition of unity and invariant property.
 
\begin{thm}\label{Thm: norm attainment theorem}
	Let $\varphi\in IP\big( C(X)^k , \mathcal{B(H)}\big),$  then $\|\varphi\|=\|\varphi(1,1, \cdots, 1)\|.$
\end{thm}

\begin{proof} 
 We prove this result for $k=3$ and $k=4$ as the same argument will follow for $k=2m$ and $k=2m-1$ for $m \geq 1$. Now, we proceed to prove for $k=3.$
 
(i): \noindent \textbf{Case 1:} Let $\varphi(1,1,1)=0$ and $f,g \in C(X)$ such that $g \geq 0$. Consider, $$\|f\|^2\varphi(1,g,1)-\varphi(f,g,\overline{f})=\varphi(1,g,\|f\|^2-f\overline{f})\geq 0.$$
Thus, $0 \leq \varphi(f,g,\overline{f}) \leq \|f\|^2\varphi(1,g,1)$. Now,
$$\|g\|\varphi(1,1,1)-\varphi(1,g,1)=\varphi(1,\|g\|-g,1)\geq 0.$$
Therefore, $\varphi(1,g,1) \leq \|g\| \varphi(1,1,1)$. This gives us $\varphi(1,g,1)=0$. So, $\varphi(f,g, \overline{f})=0$.

Let $f,g,h \in C(X)$. Then,
$$\varphi(f,g,h)=\varphi(1,g,fh)=\varphi(1,g,k)=\varphi(1,g,k_1+ik_2),$$ where $fh=k=k_1+ik_2$, $k_1, k_2 \in C(X)_{\mbox{sa}}$.
Now, 
$$\varphi(1,g,k_1+ik_2)=\varphi(1,g,k_1)+i\varphi(1,g,k_2).$$
We write $k_i=k_{i1}-k_{i2}$, where $k_{i1}, k_{i2} \geq 0$, $i=1,2$. Using the multilinearity of $\varphi$ and $\varphi(f,g,\overline{f})=0,$ we get $\varphi(f,g,h)=0$. Therefore $\|\varphi\|=\|\varphi(1,1,1)\|=0$. 

\noindent \textbf{Case 2:} Without loss of generality, let $0<\|\varphi(1,1,1)\| \leq 1$. Since $\varphi$ is invariant, we get $\|\varphi\|=\sup\{\|\varphi(1,f,g)\|:~\|f\| \leq 1,~\|g\|\leq 1,~~ f,g \in C(X)\}$. Let $f, g \in C(X)$, then for $\epsilon>0,$ there exist open covers $\{U_i\}_{i=1}^n$, $\{V_j\}_{j=1}^m$ of $X$ such that $|f(x)-f(x_i)|<\epsilon$ for all $x\in U_i$ and $|g(y)-g(y_j)|<\epsilon$ for all $y \in V_j$ and  $x_1,x_2,\ldots,x_n, y_1,y_2, \ldots, y_m \in X$. Let $\{p_i\}_{i=1}^n$ and $\{q_j\}_{j=1}^m$ are non-negative functions on $X$ satisfying $\mbox{Supp}~~ p_i \subseteq U_i$, $\mbox{Supp}~~q_j \subseteq V_j$, $\sum_{i=1}^n p_i=1$ and $\sum_{j=1}^m q_j=1$. Now,
\begin{align*}
\big\vert
f(x)-\sum_{i=1}^n f(x_i)p_i(x)\big\vert 
&= \big\vert\sum_{i=1}^n(f(x)-f(x_i))p_i(x)\big \vert \leq \epsilon \sum_{i=1}^n p_i(x)=\epsilon.
\end{align*}

Similarly, we get $\big\vert g(y)-\sum\limits_{j=1}^m g(y_j)q_j(y)\big\vert<\epsilon$ for all $y \in X$. Now,
\begin{align}\label{Equation: 6}
\notag \|\varphi(1,f,g)\| &\leq \|\varphi(1, f-\sum_{i=1}^nf(x_i)p_i,g)\|+\|\varphi(1, \sum_{i=1}^nf(x_i)p_i,g)\|  \\
&\leq \|\varphi\|\epsilon+\|\varphi(1,\sum_{i=1}^nf(x_i) p_i,g)\|.
\end{align}

Moreover, 
\begin{align}\label{Equation: 7}
	\notag &\Big\Vert \varphi(1,\sum_{i=1}^n f(x_i)p_i,g)\Big\Vert
	\leq \Big\Vert\varphi(1,\sum_{i=1}^nf(x_i)p_i,g-\sum_{j=1}^mg(y_j)q_j) \Big\Vert+\Big\Vert\sum_{j=1}^m \sum_{i=1}^n f(x_i)g(y_j)\varphi(1,p_i,q_j)\Big\Vert.
\end{align}
Notice that $\Vert\sum\limits_{i=1}^n f(x_i)p_i\Vert \leq \Vert f\Vert \sum\limits_{i=1}^{n}p_{i} \leq 1$. Thus, $\Big\Vert \varphi(1,\sum_{i=1}^n f(x_i)p_i,g-\sum_{j=1}^m g(y_j)q_j) \Big\Vert \leq \Vert \varphi \Vert \epsilon.$
Also, notice that $\sum\limits_{i,j=1}^{n,m} \varphi(1,p_i,q_j)=\varphi(1,\sum\limits_{i=1}^n p_i, \sum\limits_{j=1}^m q_j)=\varphi(1,1,1)\leq 1$ and $\left|f(x_i)\right|,\left|g(y_j)\right|<1$ for all $1 \leq i \leq n$, $1 \leq j \leq m$. By \cite{Paulsen}[Lemma 2.3], we have $\|\sum\limits_{i=1}^n\sum\limits_{j=1}^m f(x_i)g(y_j)\varphi(1,p_i,q_j)\| \leq 1$. Hence from Equation \eqref{Equation: 6} and Equation \eqref{Equation: 7}, we obtain $\|\varphi(1,f,g)\| \leq  \|\varphi \| \epsilon+\|\varphi\|\epsilon+1$. Since $\epsilon>0$ is arbitrary, we obtain $\|\varphi(1,f,g)\| \leq 1$.

Let $\|\varphi(1,1,1)\|=r \neq 0 \leq 1$. Clearly $\|\varphi(1,1,1)\| \leq \|\varphi\|$. Further, we have $\|\dfrac{1}{r}\varphi(1,1,1)\| \leq 1$. Now, using the above argument, we get $\|\dfrac{1}{r}\varphi\|\leq 1$. Hence $\|\varphi\|\leq \|\varphi(1,1,1)\|$. 

(ii): \noindent \textbf{Case 1:} Let $\varphi(1,1,1,1)=0$ and $f,g \in C(X)$. Consider $\|f\|^2\|g\|^2 \varphi(1,1,1,1)-\varphi(f,g,\overline{g},\overline{f})=\varphi(1,1,\|g\|^2,\|f\|^2)-\varphi(1,1,g\overline{g},f\overline{f})=\varphi(1,1,\|g\|^2-g\overline{g},\|f\|^2)+\varphi(1,1,g\overline{g},\|f\|^2-f\overline{f})\geq 0$. Since $\varphi(1,1,1,1)=0,$ we get $\varphi(f,g,\overline{g},\overline{f})=0$. Let $f,g,h,k \in C(X)$. Since $\varphi$ is invariant, we have $\varphi(f,g,h,k)=\varphi(1,1,gh,fk)$. Now using the same  argument of (i), we get $\varphi(f,g,h,k)=0$. Hence $\|\varphi\|=\|\varphi(1,1,1,1)\|=0$.

\noindent \textbf{Case 2:} Without loss of generality, let $0<\|\varphi(1,1,1,1)\| \leq 1$. Since $\varphi$ is invariant we get $\|\varphi\|=\sup\{\|\varphi(1,1,f,g)\Vert :~f,g \in C(X),~\|f\|,\|g\|\leq 1\}$. Let $f,g \in C(X)$, $\|f\|,\|g\|\leq 1$. Since $X$ is compact, for $\epsilon>0$, there exist finite open covers $\{U_i\}_{i=1}^n$, $\{V_j\}_{j=1}^m$ of $X$ and $x_i, y_j \in X$ for $1 \leq i \leq n$, for $1 \leq j \leq m$ such that 
$\left \vert f(x)-f(x_i)\right \vert<\epsilon$ for all $x\in U_i$, $\left \vert g(y)-g(y_j)\right \vert <\epsilon$ for all $y\in V_j$. Let $\{p_i\}_{i=1}^n$ and $\{q_j\}_{j=1}^m$ are non-negative functions on $X$ satisfying the conditions $\mbox{Supp}~~ p_i \subseteq U_i$, $\mbox{Supp}~~q_j \subseteq V_j$, $\sum\limits_{i=1}^n p_i=1$ and $\sum\limits_{j=1}^m q_j=1$.
By providing similar estimations to the case $(k=3),$ we get the following:
\begin{align*}
\|\varphi(1,1,f,g)\|  &\leq \|\varphi(1,1,f-\sum_{i=1}^nf(x_i)p_i,g)\|+\|\varphi(1,1,\sum_{i=1}^nf(x_i)p_i,g)\|\\
&\leq \|\varphi\|\epsilon+\|\varphi(1,1,\sum_{i=1}^nf(x_i)p_i,g-\sum_{j=1}^mg(y_j)q_j)\|+\|\sum_{i=1}^n\sum_{j=1}^mf(x_i)g(y_j)\varphi(1,1,p_i,q_j)\| \\
&\leq \|\varphi \| \epsilon+\|\varphi \| \epsilon +\|\sum_{i=1}^n\sum_{j=1}^mf(x_i)g(y_j)\varphi(1,1,p_i,q_j)\|\\
&\leq 2\|\varphi \| \epsilon+1.
\end{align*}
Since $\epsilon>0,$ we have $\|\varphi(1,1,f,g)\| \leq 1$. It follows that $\|\varphi\| \leq 1$. This completes the proof.
\end{proof}

We now present an example of an invariant multilinear positive map on matrix algebras.

\begin{ex}
The mapping $\varphi:~M_n(\mathbb{C})^3 \longrightarrow M_n(\mathbb{C})$ defined by 
\begin{align*}
\varphi(A_1,A_2,A_3)=\mbox{Tr}(A_1A_3)\mbox{Tr}(A_2)I,
\end{align*}
 is positive invariant multilinear map, where $\mbox{Tr}(A)$ denotes the normalized trace of $A$.
\end{ex}

The following result ensures the norm attainment of multilinear invariant positive maps on finite-dimensional $C^*$-algebras at the point $(I, I, \cdots ,I)$.

\begin{thm}\label{RDND}
Let $\varphi\in IP\big( M_n(\mathbb{C})^k , \mathcal{B(H)}\big),$  then $\|\varphi\|=\|\varphi(I,I, \cdots, I)\|.$
\end{thm}
\begin{proof}
 We show this result for $k=3$ and $k=4$ as the same line of arguments will follow for $k=2m$ and $k=2m-1$ for $m \geq 1$. Now, we proceed to prove for $k=3.$

\noindent \textbf{Case 1:} Notice that
  \begin{align*}
  \|\varphi\|&=\displaystyle\sup_{\|A_i\|\leq 1} \|\varphi(A_1,A_2,A_3)\| =\sup_{\|A_i\| \leq 1} \|\varphi(I,A_2,A_1A_3)\| ~~~~~~(\text{ as } \varphi \text{ is invariant }).
  \end{align*}
  Let $U$, $V$ are two unitary matrices, then $U=\displaystyle\sum_{i=1}^n\lambda_i P_i$, $V=\displaystyle\sum_{j=1}^m\mu_j Q_j$, where $P_i$ and $Q_j$ are projections such that $\displaystyle \sum_{i=1}^n P_i=\sum_{j=1}^m Q_j=I$. Thus, we get $\varphi(I,U,V)=\displaystyle \sum_{i,j=1}^{n,m} \lambda_i \mu_j\varphi(I,P_i,Q_j)$. Since  $\varphi$ is invariant, we have $\varphi(I,P_i,Q_j)=\varphi(Q_j,P_i,Q_j) \geq 0$. Now, $\varphi(I,I,I)=\displaystyle \sum_{i,j=1}^{n,m} \varphi(I,P_i,Q_j)$. 
  Now, it is immediate to see that
   $$\begin{bmatrix}
   \varphi(I,I,I)   & \varphi(I,U,V)\\
    \varphi(I,U,V)^*&  \varphi(I,I,I) 
  \end{bmatrix} \geq 0.$$
  Invoking  \cite[Lemma 3.1]{Paulsen} we observe that $\|\varphi(I,U,V)\| \leq \|\varphi(I,I,I)\|$. Using the singular value decomposition, we know that any matrix $A \in M_n(\mathbb{C})$ can be written as an average of two unitary matrices. So, we conclude that $\|\varphi(I,A,B)\| \leq \|\varphi(I,I,I)\|$, where $A,~B\in M_n(\mathbb{C})$. Hence $\|\varphi\|=\|\varphi(I,I,I)\|$.
 
 \noindent \textbf{Case 2:} Here $\|\varphi\|=\displaystyle\sup_{\|A_i\|\leq 1} \|\varphi(A_1,A_2,A_3,A_4)\|=\sup_{\|A_i\| \leq 1} \|\varphi(I,I,A_2A_3,A_1A_4)\|$. Following the above argument, we obtain the desired result. This completes the proof.
\end{proof}
Recall that every positive linear map between two $C^*$-algebras is symmetric. In the next result, we show that this result can be extended to invariant positive multilinear maps. We could not find this observation
 in the literature \cite{Heo 00, Heo}. 
  \begin{proposition}\label{Prop: positivity implies symmetric}
Let $\varphi\in IP\big( \mathcal{A}^k , \mathcal{B}\big).$ Then $\varphi$ is symmetric, that is, $\varphi^*=\varphi$.
\end{proposition}

\begin{proof} It is enough to prove for $k=3,4.$  

(i.) Let $\varphi\in IP\big( \mathcal{A}^3 , \mathcal{B}\big).$ Since $\varphi$ is invariant, it is enough to show that $\varphi^*(1,x_2,x_3)=\varphi(1,x_2,x_3)$, where $x_2, x_3 \in \mathcal{A}$. Let $x_2,x_3 \in \mathcal{A}_{\mbox{sa}}$. We can write $x_2=x_2^1-x_2^2$ and $x_3=x_3^1-x_3^2$, where $x_2^i, x_3^i \geq 0$, $i=1,2$. Now,
\begin{align*}
\varphi(1,x_2,x_3)^*&=\varphi(x_3^1-x_3^2,x_2^1-x_2^2,1 )^* \quad\quad (\text{Since~} \varphi \text{ is invariant})\\
&=\varphi(x_3^1,x_2^1,1)^*-\varphi(x_3^2,x_2^1,1)^*-\varphi(x_3^1,x_2^2,1)^*+\varphi(x_3^2,x_2^2,1)^*\\
&=\varphi(x_3^1,x_2^1,1)-\varphi(x_3^2,x_2^1,1)-\varphi(x_3^1,x_2^2,1)+\varphi(x_3^2,x_2^2,1)\\
&=\varphi(x_3^1-x_3^2,x_2^1-x_2^2,1)\\
&=\varphi(1,x_2,x_{3}).\quad\quad (\text{Since~} \varphi \text{ is invariant})
\end{align*}
It follows that $\varphi(1,x_2,x_3)^*=\varphi(x_3,x_2,1)^*=\varphi(1,x_2,x_3)=\varphi(x_3,x_2,1)=\varphi^*(1,x_2,x_3)$. 

Let $u, v\in \mathcal{A}$, $u=u_1+iu_2$, $v=v_1+iv_2$, where $u_i, v_i \in \mathcal{A}_{\mbox{sa}},~i=1,2$. Now,
\begin{align*}
\varphi(1,u^*,v^*)^*&=\varphi(1,u_1-iu_2,v_1-iv_2)^*\\
&=\varphi(1,u_1,v_1)^*
+i\varphi(1,u_1,v_2)^*+i\varphi(1,u_2,v_1)^*-\varphi(1,u_2,v_2)^*\\
&=\varphi^*(1,u_1,v_1)+i\varphi^*(1,u_1,v_2)+i\varphi^*(1,u_2,v_1)-\varphi^*(1,u_2,v_2)=\varphi^*(1,u,v).
\end{align*}
Therefore, we get $\varphi^*(1,u,v)=\varphi(1,u^*,v^*)^*=\varphi(1,u,v)$.\\
(ii.) Let $\varphi\in IP\big( \mathcal{A}^4 , \mathcal{B}\big).$ Since $\varphi$ is invariant, it is enough to show that $\varphi^*(1,1,x_3,x_4)=\varphi(1,1,x_3,x_4)$, where $x_3, x_4 \in \mathcal{A}$. Let $x_3,x_4 \in \mathcal{A}_{\mbox{sa}}$. We can write $x_3=x_3^1-x_3^2$ and $x_4=x_4^1-x_4^2$, where $x_3^i, x_4^i \geq 0$, $i=1,2$.  Now,
\begin{align*}
\varphi(1,1,x_3,x_4)^*
&=\varphi(1,1,x_3^1,x_4^1)^*-\varphi(1,1,x_3^1,x_4^2)^*-\varphi(1,1,x_3^2,x_4^1)^*+\varphi(1,1,x_3^2,x_4^2)^*\\
&=\varphi(1,1,x_3^1,x_4^1)-\varphi(1,1,x_3^1,x_4^2)-\varphi(1,1,x_3^2,x_4^1)+\varphi(1,1,x_3^2,x_4^2)\\
&=\varphi(1,1,x_3^1-x_3^2, x_4^1-x_4^2)\\
&=\varphi(x_4,x_3,1,1). 
\end{align*}
Thus, $\varphi(1,1,x_3,x_4)^*=\varphi(x_4,x_3,1,1)=\varphi(x_4,x_3,1,1)^*=\varphi^*(1,1,x_3,x_4)$. Let $u, v\in \mathcal{A}$, $u=u_1+iu_2$, $v=v_1+iv_2$, where $u_i, v_i \in \mathcal{A}_{\mbox{sa}},~i=1,2$. Now,
we see the following computations.
\begin{align*}
\varphi(1,1,u^*,v^*)^*
&=\varphi(1,1,u_1,v_1)^*+i\varphi(1,1,u_1,v_2)^*+ i\varphi(1,1,u_2,v_1)^*-\varphi(1,1,u_2,v_2)^*\\
&=\varphi^*(1,1,u_1,v_1)+i\varphi^*(1,1,u_1,v_2)+i\varphi^*(1,1,u_2,v_1)-\varphi^*(1,1,u_2,v_2)\\
&=\varphi^*(1,1,u,v).
\end{align*}
So we get $\varphi^*(1,1,u,v)=\varphi(1,1,u^*,v^*)^*=\varphi(1,1,u,v)$.
This completes the proof.
\end{proof}

We show that every multilinear invariant  CP map is completely bounded. This result is a multilinear version of \cite[Proposition 3.6]{Paulsen}. Possibly, it was unknown in the literature \cite{Heo 00, Heo}. 
\begin{proposition}\label{invariant CP implies CB}
Let $\varphi\in ICP\big( \mathcal{A}^k , \mathcal{B}\big).$ Then the following properties hold.
	\begin{enumerate}
		\item[(i)] If $k=3$, then $\varphi$ is completely bounded map and $\|\varphi\|_{\mbox{cb}}\leq 2^4\|\varphi(1,1,1)\|$.
		\item[(ii)] If $k=4$, then $\varphi$ is completely bounded map and $\|\varphi\|_{\mbox{cb}}\leq 2^4\|\varphi(1,1,1,1)\|$.
	\end{enumerate}
\end{proposition}

\begin{proof}
(i): Due to invariant property of $\varphi$, we can write, $$\|\varphi\|_{\mbox{cb}}=\sup\{\|\varphi_n\|:~n \in \mathbb{N}\},~~ \mbox{where}~~ \|\varphi_n\|=\sup\{\|\varphi_n(1,x,y)\|:~\|x\|,~\|y\| \leq 1\}.$$
Let $x,y \in M_n(\mathcal{A})_{\mbox{sa}}$, then we can find $x_{i}, y_{i}\geq 0,~ i=1,2$  such that $x=x_1-x_2,$ $y=y_1-y_2$ and
 $\Vert x_{i}\Vert \leq \Vert x\Vert ,$  $\Vert y_{i}\Vert \leq \Vert y\Vert$. Here,
$$\|x_i\|\|y_j\|\varphi_n(1,1,1)-\varphi_n(1,x_i,y_j)=\varphi_n(1,\|x_i\| -x_i,\|y_j\|)+\varphi_n(1,x_i, \|y_j\|-y_j)\geq 0,~~i,j=1,2.$$
So, $\varphi_n(1,x,y)=\varphi_n(1, x_1-x_2,y_1-y_2)=\varphi_n(1,x_1,y_1)+\varphi_n(1,x_2,y_2)-\varphi_n(1,x_1,y_2)-\varphi_n(1,x_2,y_1)$. Now,  $\Vert \varphi_{n}(1, x_{i}, y_{j})
\Vert \leq \Vert  x_{i}\Vert\Vert y_{j} \Vert  \Vert \varphi_{n}(1, 1, 1) \Vert 
  \leq \Vert  x\Vert\Vert y \Vert \Vert \varphi_{n}(1, 1, 1) \Vert.$ This implies that $\|\varphi_n(1,x,y)\| \leq 4 \|x\|\|y\|\|\varphi_n(1,1,1,)\|.$

Let $u,v \in M_n(\mathcal{A})$ such that $u=u_1+iu_2$ and $v=v_1+iv_2$, where $u_i, v_i \in M_n(\mathcal{A})_{\mbox{sa}}$ and $i=1,2$. Again by some simple calculations, we get $\|\varphi_n(1,u,v)\|\leq 16 \|u\|\|v\|\|\varphi_n(1,1,1)\|$. Hence $\|\varphi_n\| \leq 2^4\|\varphi_n(1,1,1)\|.$
Also, it is easy to notice that $\Vert \varphi_{n}(1, 1,1)\Vert=\Vert \varphi(1,1,1)\Vert.$ Therefore we have, $ \Vert \varphi\Vert_{cb} \leq2^{4}\Vert \varphi(1,1,1)\Vert.$

(ii): Due to invariant property of $\varphi$, we can write, $$\|\varphi\|_{\mbox{cb}}=\sup\{\|\varphi_n\|:~n \in \mathbb{N}\},~~ \mbox{where}~~ \|\varphi_n\|=\sup\{\|\varphi_n(1,1,x,y)\|:~\|x\|,~\|y\| \leq 1\}.$$
Let $x,y \in  M_n(\mathcal{A})_{\mbox{sa}}$, then $x=x_1-x_2$ and $y=y_1-y_2$, where $x_i,~y_i \geq 0,~i=1,2$. Note that
$$\|x_i\|\|y_j\|\varphi_n(1,1,1,1)-\varphi_n(1,1,x_i,y_j)=\varphi_n(1,1,\|x_i\|-x_i,\|y_j\|)+\varphi_n(1,1,x_i, \|y_j\|-y_j)\geq 0$$  for $i,j=1,2.$
So, $\varphi_n(1,1,x,y)=\varphi_n(1,1,x_1-x_2,y_1-y_2)=\varphi_n(1,1,x_1,y_1)+\varphi_n(1,1,x_2,y_2)-\varphi_n(1,1,x_1,y_2)-\varphi_n(1,1,x_2,y_1)$. By following the above arguments, we get $\|\varphi_n(1,1,x,y)\| \leq 4 \|x\|\|y\|\|\varphi_n(1,1,1,1)\|.$

Let $u,v \in M_n(\mathcal{A})$ such that $u=u_1+iu_2$ and $v=v_1+iv_2$, where $u_i, v_i \in M_n(\mathcal{A})_{\mbox{sa}},~~i=1,2$. Again by some simple calculations we obtain $\|\varphi_n(1,1,u,v)\|\leq 16 \|u\|\|v\|\|\varphi_n(1,1,1,1)\|$. It is easy to see that $\|\varphi_n\| \leq 2^4\|\varphi_n(1,1,1,1)\|= 2^4\|\varphi(1,1,1,1)\|$ for all $n\in \mathbb{N}.$ Therefore, $ \Vert \varphi\Vert_{cb} \leq2^{4}\Vert \varphi(1,1,1,1)\Vert.$
This completes the proof.
\end{proof}

 Next, we provide an example of a positive multilinear map on a commutative $C^*$-algebra that is not CP. This observation contrasts with the well-known theorem due to Stinespring \cite[Theorem 3.11]{Paulsen} that every positive map on a commutative $C^*$-algebra is automatically CP.
\begin{ex}\label{ex1}
Let $\varphi: (\mathbb{C}^2)^3 \to \mathbb{C}$ be an invariant map
 defined by $\varphi((a_1,a_2),(b_1,b_2),(c_1,c_2))=a_1b_1c_1$. Clearly $\varphi$ is multilinear. We can easily observe that $\varphi$ is positive. Now consider $\varphi_2:~M_2(\mathbb{C}^2) \times M_2(\mathbb{C}^2) \times M_2(\mathbb{C}^2) \longrightarrow M_2(\mathbb{C})$.  The $(1,1)$-th entry of
\begin{align*}
&\varphi\left(\begin{bmatrix} 
 (1,1)& &(0,0)\\
 (0,0)& &(0,0)
\end{bmatrix},
\begin{bmatrix} 
 (1,0)& &(1,0)\\
 (1,0)& &(1,0)
\end{bmatrix},
\begin{bmatrix} 
 (1,0)& &(-2,0)\\
 (-2,0)& &(4,0)
\end{bmatrix}\right),
\end{align*} is
 $-\varphi(e_1,e_1,e_1)$. This implies that  $\varphi_2$ is not positive as $\varphi(e_1,e_1,e_1)=1$.  Hence $\varphi$ is not CP.\end{ex}

Motivated by the Example~\ref{ex1}, we produce a family of invariant positive multilinear maps on a commutative domain that are not CP.

\begin{ex}
Let $X$ be a compact Hausdorff space and $x_0 \in X$. Consider $\varphi \in IP(C(X)^3, \mathbb{C})$ defined by $\varphi(f,g,h)=f(x_0)g(x_0)h(x_0)$. Following a similar type of calculation as in Example~\ref{ex1} we conclude that $\varphi$ is not CP.
\end{ex}


\section{Block multilinear CP maps} \label{Multi CP Maps}
We start this section with the definition of block multilinear CP maps.
For $n,m\geq 1$ and  fixed $k \geq 1$ set \begin{align}
M_{n}(\A)^k&:=\underbrace{M_{n}(\A)\times M_{n}(\A)\times \cdots \times M_{n}(\A)}_{k \text{-times}},\\
M_{m}\big(M_{n}(\A)\big)^k&:=\underbrace{M_{m}\big(M_{n}(\A)\big)\times M_{m}\big(M_{n}(\A)\big)\times \cdots \times M_{m}\big(M_{n}(\A)\big)}_{k \text{-times}}.
\end{align}
Let  $\varphi=[\varphi_{ij}]:M_{n}(\A)^k \to M_n(B(\mathcal{H}))$ be a multilinear map. Then the multilinear map $\varphi=[\varphi_{ij}]$ is called \emph{block} if
\begin{align}
[\varphi_{ij}]\left([a_{{1}_{ij}}], [a_{{2}_{ij}}],\ldots, [a_{{k}_{ij}}]\right):=\left[\sum_{r_1,r_2,\ldots,r_{k-1}=1}^n \varphi_{ij}(a_{{1}_{ir_1}}, a_{{2}_{r_1r_2}}, \ldots,a_{{k}_{r_{k-1}j}})\right].
\end{align}
A  block multilinear map $\varphi$ is called a \emph{positive map} if $$\varphi([a_{{1}_{ij}}],[a_{{2}_{ij}}],\ldots ,[a_{{k}_{ij}}]) \geq 0,$$ whenever $([a_{{1}_{ij}}],[a_{{2}_{ij}}],\ldots , [a_{{k}_{ij}}])=([a_{{k}_{ij}}]^*,[a_{{k-1}_{ij}}]^*,\ldots, [a_{{1}_{ij}}]^*)$ with an additional assumption that $[a_{{m}_{ij}}]$ is positive when $k=2m-1$. 
Given a multilinear map $\varphi=[\varphi_{ij}]:M_{n}(\A)^k \to M_n(\mathcal{B(H)}),$ the $t$-amplification map is a multilinear map
 $\varphi_t:M_{t}\big(M_{n}(\A)\big)^k \to M_t(M_n(B(\mathcal{H})))$ defined by 
 \begin{align}
 \varphi_t([A_{{1}_{ij}}],[A_{{2}_{ij}}],\ldots , [A_{{k}_{ij}}])=\left[\sum_{r_1,r_2,\ldots,r_{k-1}=1}^t \varphi_{ij}\left (A_{{1}_{ir_1}}, A_{{2}_{r_1r_2}}, \ldots, A_{{k}_{r_{k-1}j}}\right)
 \right]\end{align}
 for all $[A_{{1}_{ij}}],[A_{{2}_{ij}}],\ldots , [A_{{k}_{ij}}]\in M_{t}(M_{n}(\mathcal{B(H)})).$
 A block multilinear map $\varphi= [\varphi_{ij}]$ is called \emph{completely positive} (in short: CP) if $\varphi_t$ is positive for all $t \in \mathbb{N}$. 
 
   \begin{Def}
A  multilinear map $\varphi=[\varphi_{ij}]: M_{n}(\mathcal{A})^{k}\to M_{n}(\mathcal{B(H)})$ is is said to be  \emph{invariant}
 if it satisfies the following:
\begin{enumerate}
\item[(i.)]  $k=2m-1;$ for all $A_{1}, \hdots,A_{2m-1}, C_{1}, \hdots, C_{m-1} \in M_{n}(\mathcal{A}),$   $[\varphi_{ij}]$ satisfies
\begin{align*}
    [\varphi_{ij}]\big(A_{1}C_{1},\hdots, A_{m-1} C_{m-1} , A_{m}, A_{m+1}, \hdots,A_{2m-1}\big)
    =&[\varphi_{ij}]\big(A_{1},\hdots, A_{m-1}, A_{m}, C_{m-1}A_{m+1} , \hdots,C_{1}A_{2m-1}\big).
\end{align*} 

\item[(ii.)]  $k=2m;$ for all $A_{1}, \hdots,A_{2m}, C_{1}, \hdots, C_{m} \in M_{n}(\mathcal{A}),$ $[\varphi_{ij}]$ satisfies
\begin{align*}
     [\varphi_{ij}]\big(A_{1}C_{1},\hdots, A_{m} C_{m} , A_{m+1}, \hdots,A_{2m}\big)
    =&[\varphi_{ij}]\big(A_{1},\hdots, A_{m}, C_{m}A_{m+1} , \hdots,C_{1}A_{2m}\big).
\end{align*} 
\end{enumerate}
 \end{Def}
 The \emph{adjoint} of a multilinear block map $\varphi=[\varphi_{ij}]: M_{n}(\A)^k\to M_{n}(\mathcal{B(H)})$ is denoted and defined by
 $$
 \varphi^*\big(A_{1},A_{2},\ldots,A_{k}\big):=\varphi\big(A_{k}^*, A_{k-1}^*,\ldots, A_{1}^*\big)^*  
 $$
 for all  $A_{1},A_{2},\ldots,A_{k}\in M_{n}(\A).$ A block multilinear map $\varphi$ is called \emph{symmetric} if $\varphi^*=\varphi.$
 
  The next result shows that if a block map possess invariant property then each entry of the block map also possess invariant property.
   This observation is a key tool to prove Theorem~\ref{Theorem: BlockStinespring}.
 \begin{proposition}\label{invariant} Let $\varphi= [\varphi_{ij}]: M_{n}(\mathcal{A})^k \to M_{n}(\mathcal{B(H)})$ be an invariant block $k$-linear map. 
Then $\varphi_{ij}: \mathcal{A}^k \to \mathcal{B(H)}$ is an invariant $k$-linear map for $1\leq i,j \leq n.$
\end{proposition}
\begin{proof}
Without loss of generality, we assume that $k=3,4$ and $n=2$.

\noindent \textbf{Case 1 ($k=3$):} Let $A_l=[a_{l_{ij}}],~l=1,2,3$ and $B=[b_{ij}]$. As $\varphi$ is invariant, we write $\varphi(A_1B,A_2,A_3)=\varphi(A_1,A_2,BA_3)$.
Choose $a_{l_{11}}=a_l,~l=1,2,3$ and $b_{11}=b$ and other entries are zero.
 From the above equation, we obtain $\varphi_{11}(a_1b,a_2,a_3)=\varphi_{11}(a_1,a_2,ba_3)$ for all $a_1, a_2, a_3, b \in \A$. Hence $\varphi_{11}$ is invariant. Now, choose $a_{1_{12}}=a_1$, $a_{l_{22}}=a_l$, $l=2,3$, $b_{11}=b_{22}=b$ and other entries are zero. From the above equation we get, $\varphi_{12}(a_1b,a_2,a_3)=\varphi_{12}(a_1,a_2,ba_3)$ for all $a_1, a_2, a_3, b \in \A$. So $\varphi_{12}$ is invariant. Next, choose $a_{1_{21}}=a_1$, $a_{l_{11}}=a_l$, $l=2,3$, $b_{11}=b_{22}=b$ and other entries are zero. From the above equation, we conclude that $\varphi_{21}(a_1b,a_2,a_3)=\varphi_{21}(a_1,a_2,ba_3)$ for all $a_1, a_2, a_3, b \in \A$. Hence $\varphi_{21}$ is invariant. Again consider $a_{l_{22}}=a_l$,  $l=1,2,3$, $b_{22}=b$ and other entries are zero. The same equation gives us $\varphi_{22}(a_1b, a_2, a_3)=\varphi_{22}(a_1, a_2, ba_3)$ for all $a_1, a_2, a_3, b \in \A$. Hence $\varphi_{22}$ is also invariant.
 
\noindent \textbf{Case 2 ($k=4$):} Let $A_l=[a_{l_{ij}}],~l=1,2,3$ and $B_m=[b_{m_{ij}}],~m=1,2$. As $\varphi$ is invariant, we write $\varphi(A_1B_1,A_2B_2,A_3,A_4)=\varphi(A_1,A_2,B_2A_3,B_1A_4)$.
 Choose $a_{l_{11}}=a_l,~l=1,2,3,4$, $b_{m_{11}}=b_m,~m=1,2$ and other entries are zero. From the above equation we get $\varphi_{11}(a_1b_1,a_2b_2,a_3,a_4)=\varphi_{11}(a_1,a_2,b_2a_3,b_1a_4)$ for all $a_1,a_2,a_3,a_4,b_1,b_2 \in \A$. Hence $\varphi_{11}$ is invariant. Next, choose $a_{1_{12}}=a_1$, $a_{l_{22}}=a_l,~l=2,3,4$ and $b_{m_{11}}=b_{m_{22}}=b_m,~m=1,2$ and other entries are zero. From the above equation we observe $\varphi_{12}(a_1b_1,a_2b_2,a_3,a_4)=\varphi_{12}(a_1,a_2,b_2a_3,b_1a_4)$ for all $a_1,a_2,a_3,a_4,b_1,b_2 \in \A$. Hence $\varphi_{12}$ is invariant.
Again choose $a_{l_{22}}=a_l,~l=1,2,3,4$  and $b_{m_{22}}=b_m,~m=1,2$ and other entries are zero.
Similarly we get $\varphi_{22}(a_1b_1,a_2b_{2},a_3,a_4)=\varphi_{22}(a_1,a_2,b_2a_3,b_1a_4)$ for all $a_1,a_2,a_3,a_4,b_1,b_2 \in \A$. Finally, consider $a_{1_{21}}=a_1$, $a_{l_{11}}=a_l,~l=2,3,4$ and $b_{m_{11}}=b_{m_{22}}=b_m,~m=1,2$ and other entries are zero. Using the above mentioned equation we conclude that $\varphi_{21}(a_1b,a_2b,a_3,a_4)=\varphi_{21}(a_1,a_2,b_2a_3,b_1a_4)$ for all $a_1,a_2,a_3,a_4,b_1,b_2 \in \A$. Thus $\varphi_{21}$ is invariant.
This completes the proof.
\end{proof}
 \label{converse}
 In the following example, we show that  the converse of Lemma~\ref{invariant} is not true in general. In other words, if each  $3$-linear map  $\varphi_{ij}: \mathcal{A}^{3}\to \mathcal{B(H)}$ is invariant for $1\leq i,j\leq 3,$ 
 then 
 the block $3$-linear map $\varphi= [\varphi_{ij}]: M_{2}(\mathcal{A})^3\to M_{2}(\mathcal{B(H)})$ may not be  invariant.

\begin{ex}
Define $\Psi:~\mathcal{A} \times \mathcal{A} \times \mathcal{A} \longrightarrow \mathcal{B(H)}$ by $$\Psi(a,b,c)=\pi_1(b)\pi_2(ac),$$ where $\pi_1$ and  $\pi_2$ are unital $*$-homomorphisms from $\mathcal{A}$ to $\mathcal{B(H)}$. 
Notice that  $\Psi$ is invariant and clearly $\Psi(1,1,1)=\pi_1(1)\pi_2(1) \neq 0$.
Now, we consider the block map $\varphi=[\varphi_{ij}]:M_2(\mathcal{A})\times M_2(\mathcal{A})\times M_2(\mathcal{A})\longrightarrow M_2(\mathcal{B(H)})$ , where $\varphi_{ij}=\Psi$ for all $1 \leq i,j \leq 2$.
If possible, let $\varphi$ is invariant. Then for $b,c\in \A$ with $b\neq 0,$ we notice that

\begin{align*}
\varphi\left(\begin{bmatrix} 
 1& &0\\
 0& &0
\end{bmatrix}
\begin{bmatrix} 
 0& &b\\
 b& &c
\end{bmatrix},
\begin{bmatrix} 
 1& &1\\
 2& &1
\end{bmatrix},
\begin{bmatrix} 
 1& &3\\
 0& &4
\end{bmatrix}\right)
= \varphi\left(\begin{bmatrix} 
 1& &0\\
 0& &0
\end{bmatrix},
\begin{bmatrix} 
 1& &1\\
 2& &1
\end{bmatrix},
\begin{bmatrix} 
 0& &b\\
 b& &c
\end{bmatrix}
\begin{bmatrix} 
 1& &3\\
 0& &4
\end{bmatrix}\right).
\end{align*}

From the above equation, we get, 
\begin{align*}
&\begin{bmatrix} 
 2b\Psi(1,1,1)& &10b \Psi(1,1,1)\\
 0& &0
\end{bmatrix}
= \begin{bmatrix} 
b \Psi(1,1,1)& &\Psi(1,1,7b+4c)\\
 0& &0
\end{bmatrix}.
\end{align*}
By comparing $(1,1)$-entries of the above matrix equation, we get $2b \Psi(1,1,1)=b\Psi(1,1,1)$. It follows that $\Psi(1,1,1) = 0,$ which is a contradiction. Therefore  $\varphi$ can not be invariant.
\end{ex}
\subsection{Stinespring theorem for invariant block multilinear CP maps}
Now we prove Stinespring theorem for block multilinear invariant CP maps.
\begin{thm}\label{Theorem: BlockStinespring} Let $\varphi= [\varphi_{ij}]\in ICP\big( M_{n}(\mathcal{A})^k , M_{n}(\mathcal{B(H)}\big).$ Then there exist a $m$-commuting family of unital  $*$-homomorphisms ${\pi^{\varphi}_{i}:\A\to \mathcal{B(K)}}$ for some Hilbert space $\mathcal{K}$ and $V_{i}\in \mathcal{B(H, K)} , 1\leq i\leq n$  such that
\begin{enumerate}
    \item[(i)] if $k = 2m-1$, then $\varphi_{ij}$ is of the form:
 $$ \varphi_{ij}(a_{1}, \hdots, a_{m-1}, a_{m}, a_{m+1}, \hdots, a_{2m-1}) = V_{i}^{\ast}\pi^{\varphi}_{1}(a_{m})\pi^{\varphi}_{2}(a_{m-1}a_{m+1}) \cdots \pi^{\varphi}_{m}(a_{1}a_{2m-1})V_{j},$$
 for all $a_{1}, a_{2}, \hdots, a_{2m-1} \in \mathcal{A}$.
    \item[(ii)] if $k = 2m$, then $\varphi_{ij}$ is of the form:
    \begin{equation*}
        \varphi_{ij}(a_{1}, \hdots, a_{m-1}, a_{m}, a_{m+1}, \hdots, a_{2m}) = V_{i}^{\ast}\pi^{\varphi}_{1}(a_{m}a_{m+1})\pi^{\varphi}_{2}(a_{m-1}a_{m+2}) \cdots \pi^{\varphi}_{m}(a_{1}a_{2m})V_{j},
    \end{equation*}
    for all $a_{1}, a_{2}, \hdots, a_{2m} \in \mathcal{A}$.
\end{enumerate}
\end{thm}
 Here the main difficulty to prove the above theorem is to find a proper sesquilinear form. The following result is the key ingredient to prove Theorem \ref{Theorem: BlockStinespring}. 
 
 \begin{lemma}\label{Lamma-Block multilienar Stinespring} Let $\varphi= [\varphi_{ij}]\in ICP( M_{n}(\mathcal{A})^k , M_{n}(\mathcal{B(H)}).$ Then there exists a semi-inner product $\langle . , . \rangle$ $\mathcal{A}^{\otimes m}\otimes
  \mathcal{H}^n  \times \mathcal{A}^{ \otimes m}\otimes  \mathcal{H}^n \to \mathbb{C}$ associate to $\varphi$ defined by
 \begin{align*}\label{kernel}
&\left \langle \sum_{l=1}^ta_{1_l} \otimes a_{2_l}\otimes\hdots \otimes a_{m_l}\otimes
 \begin{bmatrix}
 f_{l1}\\
 \vdots\\
 f_{ln}\\
\end{bmatrix}, \sum_{r=1}^t b_{1_r}\otimes b_{2_r}\otimes \hdots \otimes  b_{m_r}\otimes
 \begin{bmatrix}
 g_{r1}\\
 \vdots\\
 g_{rn}\\
\end{bmatrix} \right \rangle\\
:=&\left\{ \begin{array}{cc}
    \sum\limits_{l,r=1}^t\sum\limits_{i,j=1}^n  \big \langle \varphi_{ij}(b_{m_r}^*,\hdots ,b_{2_r}^*, b_{1_r}^*a_{1_l}, a_{2_l}, \hdots , a_{m_l})f_{lj} ,g_{ri} \big \rangle  &\text{if} \; k = 2m-1 \\
   &\\
    \sum\limits_{l,r=1}^t\sum\limits_{i,j=1}^n  \big \langle \varphi_{ij}(b_{m_r}^*,\hdots ,b_{2_r}^*, b_{1_r}^*, a_{1_l}, a_{2_l}, \hdots , a_{m_l})f_{lj} ,g_{ri} \big \rangle&\text{if} \; k = 2m, \\
\end{array} \right. 
\end{align*}
for all
 $\sum\limits_{l=1}^ta_{1_l}\otimes a_{2_l}\otimes\hdots \otimes a_{m_l}\otimes
 \begin{bmatrix}
 f_{l1}\\
 \vdots\\
 f_{lk}\end{bmatrix}, \sum\limits_{r=1}^t b_{1_r}\otimes b_{2_r}\otimes \hdots \otimes b_{m_r}\otimes
 \begin{bmatrix}
 g_{r1}\\
 \vdots\\
 g_{rk}\\
\end{bmatrix}\in {\A^{\otimes}}^m\otimes \mathcal{H}^{n}.$
\end{lemma}
\begin{proof} It is enough to show that  $\langle, . ,\rangle$ is positive semi-definite. To show this, we notice that
  \begin{align}
\notag &\Big \langle \sum_{l=1}^t a_{1_l}\otimes a_{2_l}\otimes\hdots \otimes a_{m_l}\otimes
 \begin{bmatrix}
 f_{l1}\\
 \vdots\\
 f_{lk}\end{bmatrix}, 
 \sum_{r=1}^t a_{1_r}\otimes a_{2_r}\otimes \hdots \otimes a_{m_r}\otimes
 \begin{bmatrix}
 f_{r1}\\
 \vdots\\
 f_{rk}\\
\end{bmatrix} \Big\rangle\\
\notag =&\left\{ \begin{array}{cc}
    \sum\limits_{l,r=1}^t \sum\limits_{i=1}^{n} \Big \langle \sum\limits_{j=1}^{n} \varphi_{ij}(a_{m_r}^*,\hdots, a_{2_r}^{\ast}, a_{1_r}^*a_{1_l}, a_{2_l}, \hdots, a_{m_l})f_{lj},\; f_{ri} \Big\rangle  &\text{if} \; k = 2m-1 \\
     \sum\limits_{l,r=1}^t \sum\limits_{i=1}^{n} \Big\langle \sum\limits_{j=1}^{n}\varphi_{ij}(a_{m_r}^*,\hdots, a_{2_r}^{\ast}, a_{1_r}^*, a_{1_l}, a_{2_l}, \hdots, a_{m_l})f_{lj},\; f_{ri} \Big\rangle &\text{if} \; k = 2m \\
\end{array} \right.\\
 =&\left\{ \begin{array}{cc}
    \sum\limits_{l,r=1}^t\Bigg\langle \Bigg[
     \varphi_{ij}(a_{m_r}^*, \hdots, a_{2_r}^*, a_{1_r}^*a_{1_l}, a_{2_l}, \hdots, a_{m_l})\Bigg]_{i,j=1}^{n} \begin{bmatrix}f_{l1}\\\vdots\\ f_{ln} \end{bmatrix}, \begin{bmatrix}f_{r1}\\\vdots\\ f_{rn} \end{bmatrix}\Bigg\rangle
      &\text{if}\; \; k = 2m-1 \\
\notag    &\\
\sum\limits_{l,r=1}^t    \Bigg\langle \Bigg[
     \varphi_{ij}(a_{m_r}^*, \hdots, a_{2_r}^*, a_{1_r}^*, a_{1_l}, a_{2_l}, \hdots, a_{m_l})\Bigg]_{i,j=1}^{n} \begin{bmatrix}f_{l1}\\\vdots\\ f_{ln} \end{bmatrix}, \begin{bmatrix}f_{r1}\\\vdots\\ f_{rn} \end{bmatrix}\Bigg\rangle
      &\text{if} \;\; k = 2m \\
\end{array} \right.\\
 =& \left\{ \begin{array}{cc}
     \sum\limits_{l,r,=1}^{t} \Big\langle  \varphi(D_{m_r}^{\ast}, \hdots,D_{2_r}^{\ast}, R_{1_r}^{\ast}R_{1_l}, D_{2_l}, \hdots, D_{m_l})f_{l}, \; f_{r} \Big\rangle &\text{if}  \; k = 2m-1 \\
     \sum\limits_{l,r,=1}^{t} \Big\langle  \varphi(D_{m_r}^{\ast}, \hdots,D_{2_r}^{\ast}, R_{1_r}^{\ast}, R_{1_l}, D_{2_l}, \hdots, D_{m_l})f_{l}, \; f_{r} \Big\rangle&\text{if}  \; k = 2m \\
\end{array} \right.
\end{align}
and for each $1\leq p\leq m, 1\leq l,r\leq t,$ $f_{l}, D_{p_l}$ and $R_{1_l}$ are defined by 
\begin{equation} \label{Eq: Essential}
    f_{l}= 
\begin{bmatrix}
f_{l1}\\
\vdots\\
f_{ln}
\end{bmatrix}, \; D_{p_r}= \begin{bmatrix}
a_{p_r}& 0 &\hdots &0\\
0& a_{p_r}&\hdots &0\\
\vdots&\vdots&\ddots&\vdots\\
0&0&\hdots &a_{p_r}
\end{bmatrix} \; \text{and}\; R_{1_l}= 
\begin{bmatrix}
a_{1_l}& a_{1_l}\hdots& a_{1_l}\\
0     &\hdots       &0\\
\vdots&\ddots&\vdots\\
0     &\hdots       &0\\
\end{bmatrix}.
\end{equation}
Now from Equation \eqref{kernel}, we can rewrite as
\begin{align}\label{Eq: phi is cp}
\notag&\left \langle \sum_{l=1}^ta_{1_l} \otimes a_{2_l}\otimes\hdots \otimes a_{r_l}\otimes
 \begin{bmatrix}
 f_{l1}\\
 \vdots\\
 f_{ln}\\
\end{bmatrix}, \sum_{r=1}^t a_{1_r}\otimes a_{2_r}\otimes \hdots \otimes  a_{m_r}\otimes
 \begin{bmatrix}
 f_{r1}\\
 \vdots\\
 f_{rn}\\
\end{bmatrix} \right \rangle\\
 =& \left\{ \begin{array}{cc}
    \Big\langle  \varphi_{t}(D_{m}^{\ast}, \hdots,D_{2}^{\ast}, R_{1}^{\ast}R_{1}, D_{2}, \hdots, D_{m})f, \; f \Big\rangle &\text{if} \; k = 2m-1 \\
     \Big\langle  \varphi_{t}(D_{m}^{\ast}, \hdots,D_{2}^{\ast}, R_{1}^{\ast}, R_{1}, D_{2}, \hdots, D_{m})f, \; f \Big\rangle&\text{if} \; k = 2m, \\
\end{array} \right.
\end{align}
 where 
 
 \begin{align}\label{Eq: Dp, R1}
 f= \begin{bmatrix}
 f_{1}\\
\vdots\\
f_{t}
 \end{bmatrix},\quad
  D_{p}= \begin{bmatrix}
D_{p_1}& 0 &\hdots &0\\
0& D_{p_2}&\hdots &0\\
\vdots&\vdots&\ddots&\vdots\\
0&0&\hdots &D_{p_t}
\end{bmatrix} \text{ and } R_{1}= \begin{bmatrix}
R_{1_1}& R_{1_2}\hdots& R_{1_t}\\
0     &\hdots       &0\\
\vdots&\ddots&\vdots\\
0     &\hdots       &0\\
\end{bmatrix}.\end{align}
Since $\varphi$ is CP, from Equation \eqref{Eq: phi is cp}, we conclude that $\langle . , .\rangle$ is positive semi-definite.
\end{proof}
Now, we prove Theorem \ref{Theorem: BlockStinespring} using the developed machinery.

\begin{proof}[{\bf Proof of Theorem \ref{Theorem: BlockStinespring}}] Let us define a closed set $\mathcal{N}_{\varphi}=\{x\in \mathcal{A}^{\otimes m}\otimes \mathcal{H}^n : \langle x, x \rangle =0 \}.$ Then we can construct a Hilbert space as $\mathcal{K}^{\varphi}: = 
\overline{(\mathcal{A}^{\otimes m}\otimes \mathcal{H}^n )/{\mathcal{N}_{\varphi}}} .$ For $1\leq p\leq m$ and  $a\in \mathcal{A},$ we define a linear map
$\pi^{\varphi}_{p}(a):(\mathcal{A}^{\otimes m}\otimes \mathcal{H}^n)/\mathcal{N}_{\varphi} \to (\mathcal{A}^{\otimes m}\otimes \mathcal{H}^n)/\mathcal{N}_{\varphi}$ by
\begin{align}\label{Eq: norm estimate of pi map}
  \notag \pi_{p}^{\varphi}(a)\Big(\sum_{l=1}^t a_{1_{l}}\otimes\hdots\otimes a_{p_{l}}\otimes & \hdots\otimes a_{m_{l}}\otimes 
  \begin{bmatrix}f_{l1}\\
  \vdots\\
  f_{ln}
  \end{bmatrix}
  +\mathcal{N}_{\varphi} \Big)\\
 &=\sum_{l=1}^t a_{1_{l}}\otimes\hdots\otimes aa_{p_{l}}\otimes   \hdots\otimes a_{m_{l}}\otimes 
  \begin{bmatrix}f_{l1}\\
  \vdots\\
  f_{ln}
  \end{bmatrix}
  +\mathcal{N}_{\varphi} 
  \end{align}
  for all $a_{p_{l}} \in \mathcal{A},\; f_{lj} \in \mathcal{H},\; l \in \{1,\hdots, t\}.$ Using similar computation as in  Lemma \ref{Lamma-Block multilienar Stinespring}, notice that
  \begin{align} \label{Eq: computation of norm}
 \notag  &\Big\| \pi_{p}^{\varphi}(a)\big(\sum_{l=1}^t a_{1_l}\otimes \hdots \otimes a_{p_l}\otimes \hdots \otimes a_{m_l}\otimes 
\begin{bmatrix} 
f_{l1}\\
\vdots\\
f_{ln}
\end{bmatrix}
 +\mathcal{N}_{\varphi}\big)\Big\|^2 \\
 & =\left\{ \begin{array}{cc}
     \sum\limits_{l,r,=1}^{t} \Big\langle  \varphi(D_{m_r}^{\ast}, \hdots, D_{p_r}^{\ast}\hdots, R_{1_r}^{\ast}R_{1_l}, \hdots, (a^*a\otimes I_{n})D_{p_l}, \hdots, D_{m_l})f_{l}, \; f_{r} \Big\rangle &\text{if} \; k = 2m-1 \\
     \sum\limits_{l,r,=1}^{t} \Big\langle  \varphi(D_{m_r}^{\ast}, \hdots,D_{p_r}^{\ast}, \hdots ,R_{1_r}^{\ast}, R_{1_l}, \hdots ,(a^*a\otimes I_{n})D_{p_l}, \hdots, D_{m_l})f_{l}, \; f_{r} \Big\rangle&\text{if} \; k = 2m \\
\end{array} \right.
\end{align}
where $f_{l}, D_{p_l}$ and $R_{1_l}$ are defined in Equation 
 \eqref{Eq: Essential}. Now,  we can write Equation \eqref{Eq: computation of norm} as, 
 \begin{align} \label{Eq: order relation on phi}
 \notag  &\Big\| \pi_{p}^{\varphi}(a)\big(\sum_{l=1}^t a_{1_l}\otimes \hdots \otimes a_{p_l}\otimes \hdots \otimes a_{m_l}\otimes 
\begin{bmatrix} 
f_{l1}\\
\vdots\\
f_{ln}
\end{bmatrix}
 +\mathcal{N}_{\varphi}\big)\Big\|^2 \\
  =& \left\{ \begin{array}{cc}
    \Big\langle  \varphi_{t}(D_{m}^{\ast}, \hdots,D_{p}^{\ast},\hdots ,R_{1}^{\ast}R_{1}, \hdots \big((a^*a\otimes I_{n})\otimes I_{t}\big) D_{p}, \hdots, D_{m})f, \; f \Big\rangle &\text{if} \; k = 2m-1 \\
   \Big\langle  \varphi_{t}(D_{m}^{\ast},\hdots,D_{p}^{\ast},\hdots ,R_{1}^{\ast}, R_{1}, \hdots \big((a^*a\otimes I_{n})\otimes I_{t}\big) D_{p}, \hdots, D_{m})f, \; f \Big\rangle  &\text{if} \; k = 2m \\
\end{array} \right.
 \end{align}

where $D_{l}, R_{1}$ are defined in Equation \eqref{Eq: Dp, R1}.
As $\varphi$ is CP map and invariant, we conclude that
\begin{align*}
&\left\{ \begin{array}{cc}
    \Big\langle  \varphi_{t}(D_{m}^{\ast}, \hdots,D_{p}^{\ast},\hdots ,R_{1}^{\ast}R_{1}, \hdots \big( (a^*a\otimes I_{n})\otimes I_{t}\big)D_{p}, \hdots, D_{m})f, \; f \Big\rangle &\text{if} \; k = 2m-1 \\ 
   \Big\langle  \varphi_{t}(D_{m}^{\ast},\hdots,D_{p}^{\ast},\hdots ,R_{1}^{\ast}, R_{1}, \hdots \big((a^*a\otimes I_{n})\otimes I_{t}\big)D_{p}, \hdots, D_{m})f, \; f \Big\rangle  &\text{if} \; k = 2m \\
\end{array} \right. \\
\leq &\Vert a\Vert^2  \left\{ \begin{array}{cc}
    \Big\langle  \varphi_{t}(D_{m}^{\ast}, \hdots,D_{p}^{\ast},\hdots ,R_{1}^{\ast}R_{1}, D_{p}, \hdots, D_{m})f, \; f \Big\rangle &\text{if} \; k = 2m-1 \\ 
   \Big\langle  \varphi_{t}(D_{m}^{\ast},\hdots,D_{p}^{\ast},\hdots ,R_{1}^{\ast}, R_{1}, D_{p}, \hdots, D_{m})f, \; f \Big\rangle  &\text{if} \; k = 2m .\\
\end{array} \right. 
\end{align*}
Putting $a=1$ in Equation \eqref{Eq: order relation on phi}, we obtain that
\begin{align*}
  &\left\{ \begin{array}{cc}
    \Big\langle  \varphi_{t}(D_{m}^{\ast}, \hdots,D_{p}^{\ast},\hdots ,R_{1}^{\ast}R_{1}, D_{p}, \hdots, D_{m})f, \; f \Big\rangle &\text{if} \; k = 2m-1 \\ 
   \Big\langle  \varphi_{t}(D_{m}^{\ast},\hdots,D_{p}^{\ast},\hdots ,R_{1}^{\ast}, R_{1}, D_{p}, \hdots, D_{m})f, \; f \Big\rangle  &\text{if} \; k = 2m \\
\end{array} \right. \\
=& \Big \Vert \Big(\sum_{l=1}^t a_{1_{l}}\otimes\hdots\otimes a_{p_{l}}\otimes  \hdots\otimes a_{m_{l}}\otimes 
  \begin{bmatrix}f_{l1}\\
  \vdots\\
  f_{ln}
  \end{bmatrix}
  +\mathcal{N}_{\varphi} \Big)\Big\Vert^2.
\end{align*}
Combining all the inequalities, we get 
\begin{align*}
&\Big\| \pi_{p}^{\varphi}(a)\big(\sum_{l=1}^t a_{1_l}\otimes \hdots \otimes a_{p_l}\otimes \hdots \otimes a_{m_l}\otimes 
\begin{bmatrix} 
f_{l1}\\
\vdots\\
f_{ln}
\end{bmatrix}
 +\mathcal{N}_{\varphi}\big)\Big\| \\
 \leq&
 \Vert a \Vert\Big  \Vert \sum_{l=1}^t a_{1_l}\otimes \hdots \otimes a_{p_l}\otimes \hdots \otimes a_{m_l}\otimes 
\begin{bmatrix} 
f_{l1}\\
\vdots\\
f_{ln}
\end{bmatrix}
 +\mathcal{N}_{\varphi} \Big \Vert . 
\end{align*}
Now $\pi_{p}^{\varphi}(a)$ can be extended uniquely to $\mathcal{K}_{\varphi}$ with preserving its norm as $(\mathcal{A}^{\otimes m}\otimes \mathcal{H}^n )/{\mathcal{N}_{\varphi}}$ is dense in $\mathcal{K}_{\varphi}.$  
Thus $\Vert \pi^{\varphi}_{p}(a)\Vert \leq \Vert a \Vert$ for all $a\in \A.$  It is immediate to see $\pi^{\varphi}_{p}(ab)= \pi^{\varphi}_{p}(a)\pi^{\varphi}_{p}(b)$ for all $a,b\in \A$ and $\pi^{\varphi}_{p}\pi^{\varphi}_{q}=\pi^{\varphi}_{q}\pi^{\varphi}_{p}$ for all
 $1\leq p\neq q \leq m.$ Next, we claim that
$\pi^{\varphi}_{p}$ is a $*$-homomorphism. To see this, let us observe the following:
   \begin{align*}
&\Big \langle \pi_{p}^{\varphi}(a) \sum_{l=1}^t a_{1_l}\otimes \hdots \otimes a_{p_l}\otimes \hdots \otimes a_{m_l}\otimes 
\begin{bmatrix} 
f_{l1}\\
\vdots\\
f_{ln}
\end{bmatrix}
 +\mathcal{N}_{\varphi}, \sum_{r=1}^t b_{1_r}\otimes \hdots \otimes b_{p_r}\otimes \hdots \otimes b_{m_r}\otimes 
\begin{bmatrix} 
g_{r1}\\
\vdots\\
g_{rn}
\end{bmatrix}
 +\mathcal{N}_{\varphi}  \Big\rangle\\
 =&\left\{ \begin{array}{cc}
    \sum\limits_{l,r=1}^t\sum\limits_{i,j=1}^n  \big \langle \varphi_{ij}(b_{m_r}^*,\hdots ,b_{p_r}^*, \hdots,b_{1_r}^*a_{1_l}, \hdots ,aa_{p_l}, \hdots , a_{m_l})f_{lj} ,g_{ri} \big \rangle  &\text{if} \; k = 2m-1 \\
   &\\
    \sum\limits_{l,r=1}^t\sum\limits_{i,j=1}^n  \big \langle \varphi_{ij}(b_{m_r}^*,\hdots , b_{p_r}^*,\hdots , b_{1_r}^*, a_{1_l}, \hdots, aa_{p_l}, \hdots , a_{m_l})f_{lj} ,g_{ri} \big \rangle&\text{if} \; k = 2m. \\
\end{array} \right. \\
\end{align*}
Using Proposition \ref{invariant}, we get  
\begin{align*}
&\left\{ \begin{array}{cc}
    \sum\limits_{l,r=1}^t\sum\limits_{i,j=1}^n  \big \langle \varphi_{ij}(b_{m_r}^*,\hdots ,b_{p_r}^*, \hdots,b_{1_r}^*a_{1_l}, \hdots ,aa_{p_l}, \hdots , a_{m_l})f_{lj} ,g_{ri} \big \rangle  &\text{if} \; k = 2m-1 \\
   &\\
    \sum\limits_{l,r=1}^t\sum\limits_{i,j=1}^n  \big \langle \varphi_{ij}(b_{m_r}^*,\hdots , b_{p_r}^*,\hdots , b_{1_r}^*, a_{1_l}, \hdots, aa_{p_l}, \hdots , a_{m_l})f_{lj} ,g_{ri} \big \rangle&\text{if} \; k = 2m \\
\end{array} \right. \\
 =&\left\{ \begin{array}{cc}
    \sum\limits_{l,r=1}^t\sum\limits_{i,j=1}^n  \big \langle \varphi_{ij}(b_{m_r}^*,\hdots ,(a^*b_{p_r})^*, \hdots,b_{1_r}^*a_{1_l}, \hdots ,a_{p_l}, \hdots , a_{m_l})f_{lj} ,g_{ri} \big \rangle  &\text{if} \; k = 2m-1 \\
   &\\
    \sum\limits_{l,r=1}^t\sum\limits_{i,j=1}^n  \big \langle \varphi_{ij}(b_{m_r}^*,\hdots , (a^*b_{p_r})^*,\hdots , b_{1_r}^*, a_{1_l}, \hdots, a_{p_l}, \hdots , a_{m_l})f_{lj} ,g_{ri} \big \rangle&\text{if} \; k = 2m\\
\end{array} \right. \\
=&\Big \langle  \sum\limits_{l=1}^t a_{1_l}\otimes \hdots \otimes a_{p_l}\otimes \hdots \otimes a_{m_l}\otimes 
\begin{bmatrix} 
f_{l1}\\
\vdots\\
f_{ln}
\end{bmatrix}
 +\mathcal{N}_{\varphi}, \pi_{p}^{\varphi}(a^*)\sum_{r=1}^t b_{1_r}\otimes \hdots \otimes b_{p_r}\otimes \hdots \otimes b_{m_r}\otimes 
\begin{bmatrix} 
g_{r1}\\
\vdots\\
g_{rn}
\end{bmatrix}
 +\mathcal{N}_{\varphi}  \Big\rangle.
 \end{align*}
 Now, applying the fact that $(\mathcal{A}^{\otimes m}\otimes \mathcal{H}^n )/{\mathcal{N}_{\varphi}}$ is dense in $\mathcal{K}_{\varphi}$ on the above equation, we obtain
 $\pi_{p}^{\varphi}(a^*)= \pi_{p}^{\varphi}(a)^*$ for all $a\in \A.$
 We define $ V_{{j}}: \mathcal{H}\to \mathcal{K}$ by
$
V_{{j}}(f)= 1\otimes\hdots\otimes1\otimes \begin{bmatrix}
0\\ \vdots\\ f_{j}\\\vdots \\0
\end{bmatrix} + \mathcal{N}_{\varphi}
\text{~for all } f\in \mathcal{H}. 
$
Then, $
\Vert V_{{j}}(f)\Vert^2 
=\langle \varphi_{jj}(1,\hdots ,1)f,f\rangle \leq \Vert \varphi_{jj}(1,\hdots ,1)\Vert \Vert f\Vert^2.
$
Therefore, $\Vert V_{j} \Vert \leq \Vert \varphi_{jj}(1,\hdots,1) \Vert^{\frac{1}{2}}.$ Moreover, it is immediate to see from the last computations that $V_{{jj}}$ is an isometry, whenever  $\varphi_{jj}(1,\hdots,1)=1$.  
 Finally, let $a\in \A,$  then for all $f,g\in \mathcal{H},$ we notice that 
\begin{align*}
&=\left\{ \begin{array}{cc}
   \Big\langle V_{i}^*\pi^{\varphi}_{1}(a_{m})\pi^{\varphi}_{2}(a_{m-1}a_{m+1}) \cdots \pi^{\varphi}_{m}(a_{1}a_{2m-1})V_{{j}}f,g \Big\rangle &\text{if} \; k = 2m-1 \\
   &\\
\Big    \langle V_{i}^{\ast}\pi^{\varphi}_{1}(a_{m}a_{m+1})\pi^{\varphi}_{2}(a_{m-1}a_{m+2}) \cdots \pi^{\varphi}_{m}(a_{1}a_{2m})V_{j}f, g \Big \rangle &\text{if} \; k = 2m \\
\end{array} \right.\\
 &=\left\{ \begin{array}{cc}
  \left \langle a_{m}\otimes a_{m-1}a_{m+1}\otimes\hdots\otimes a_{1}a_{2m-1}\otimes 
\begin{bmatrix}
0 \\ \vdots\\ f_{j}\\\vdots \\0 
\end{bmatrix}+\mathcal{N}_{\varphi}, 1\otimes\hdots\otimes1\otimes 
\begin{bmatrix}
0 \\ \vdots\\ g_{i}\\\vdots \\0 
\end{bmatrix}+\mathcal{N}_{\varphi}\right\rangle  &\text{if} \; k = 2m-1 \\
   &\\
 \left \langle a_{m}a_{m+1}\otimes a_{m-1}a_{m+2}\otimes\hdots\otimes a_{1}a_{2m}\otimes 
\begin{bmatrix}
0 \\ \vdots\\ f_{j}\\\vdots \\0 
\end{bmatrix}+\mathcal{N}_{\varphi}, 1\otimes\hdots\otimes1\otimes 
\begin{bmatrix}
0 \\ \vdots\\ g_{i}\\\vdots \\0 
\end{bmatrix}+\mathcal{N}_{\varphi}\right\rangle &\text{if} \; k = 2m \\
\end{array} \right.\\
&=\left\{ \begin{array}{cc}
     \big \langle \varphi_{ij}(1,\hdots ,1, \hdots,a_{m},a_{{m-1}}a_{{m+1}} \hdots , \hdots , a_{1}a_{2m-1} ) f , g \big \rangle  &\text{if} \; k = 2m-1 \\
   &\\
     \big \langle \varphi_{ij}(1,\hdots , 1,\hdots , a_{m}a_{m+1}, a_{m-1}a_{m+2}, \hdots,  a_{1}a_{2m})f , g \big \rangle &\text{if} \; k = 2m. \\
\end{array} \right. 
\end{align*}
An appeal to Proposition \ref{invariant}, we conclude that
\begin{align*}
&\left\{ \begin{array}{cc}
     \big \langle \varphi_{ij}(1,\hdots ,1, \hdots,a_{m},a_{{m-1}}a_{{m+1}} \hdots , \hdots , a_{1}a_{2m-1} ) f , g \big \rangle  &\text{if} \; k = 2m-1 \\
   &\\
     \big \langle \varphi_{ij}(1,\hdots , 1,\hdots , a_{m}a_{m+1}, a_{m-1}a_{m+2}, \hdots,  a_{1}a_{2m})f , g \big \rangle &\text{if} \; k = 2m \\
\end{array} \right. \\
=&\left\{ \begin{array}{cc}
     \big \langle \varphi_{ij}(a_{1}, \hdots, a_{m-1}, a_{m}, a_{m+1}, \hdots, a_{2m-1}) f ,g \big \rangle  &\text{if} \; k = 2m-1 \\
   &\\
     \big \langle \varphi_{ij}(a_{1}, \hdots, a_{m-1}, a_{m}, a_{m+1}, \hdots, a_{2m})f ,g \big \rangle &\text{if} \; k = 2m. \\
\end{array} \right. 
\end{align*}
This completes the proof.
\end{proof}

\subsection{Minimality condition}
 
 The concept of minimality has received significant attention in research and has proven to be highly applicable across various domains (see  \cite{W. Arveson69,BGS21,Heo10, Joita}). Hence, it is both intuitive and imperative to explore the minimality condition in the context of invariant multilinear block CP maps. In this regard, we introduce the minimality condition as a fundamental aspect of our investigation.
 
\begin{Def}[Minimality condition] Let $\varphi = [\varphi_{ij}]\in ICP\left (M_{n}(\mathcal{A})^{k},  M_{n}(\mathcal{B(H)})\right).$ Then a triple $\left (\{\pi_{i}\}_{i=1}^m, \{V_{i}\}_{i=1}^{n}, \mathcal{K}\right)$ is called minimal Stinespring tripple of $\varphi$ if 
\begin{align}\label{minimality condition}
\mathcal{K}= \overline{{\rm span}}\left\{\pi_{1}(\mathcal{A})\cdots \pi_{m}(\mathcal{A}) \big(\sum_{i=1}^{n} V_{i}\mathcal{H}\big )\right\}.
\end{align}
\end{Def}
Now, we turn our attention to the uniqueness of the Stinespring representation for block multilinear invariant CP maps.
\begin{thm}[Uniqueness property]\label{Theorem: Minimal BlockStinespring} Let $\left(\{\pi^{(p)}_{i}\}_{i=1}^m, \{V^{(p)}_{i}\}_{i=1}^{n}, \mathcal{K}^{(p)}\right), p=1,2$ be two minimal Stinespring triples of $\varphi= [\varphi_{ij}]\in ICP\left(M_{n}(\mathcal{A})^{k}, M_{n}(\mathcal{B(H)})\right).$ Then there exists a unitary operator $U: \mathcal{K}^{(1)}\to \mathcal{K}^{(2)}$ such that 
\begin{enumerate}
\item $\pi_{1}^{(2)}(a_{1}) \cdots \pi_{m}^{(2)}(a_{m})U= U\pi_{1}^{(1)}(a_{1}) \cdots \pi_{m}^{(1)}(a_{m})$ for all $a_{1}, \hdots, a_{m}\in \mathcal{A}.$
\item $UV^{(1)}_{i}=V_{i}^{(2)}~$ for all $1\leq i\leq m.$
\end{enumerate}
\end{thm}
\begin{proof} We  define a linear map $$U:  {\rm span}\left\{\pi^{(1)}_{1}(\mathcal{A})\cdots \pi^{(1)}_{m}(\mathcal{A}) \big(\sum_{j=1}^{n} V^{(1)}_{j}\mathcal{H}\big )\right\} \to  {\rm span}\left\{\pi^{(2)}_{1}(\mathcal{A})\cdots \pi^{(2)}_{m}(\mathcal{A}) 
\big(\sum_{j=1}^{n} V^{(2)}_{j}\mathcal{H}\big )\right\}$$ 
by
$$
U\left ( \Big( \sum_{l=1}^{t} \pi^{(1)}_{1}(a_{l_{1}})\cdots \pi^{(1)}_{m}(a_{l_{m}})\Big) (\sum_{j=1}^{n} V^{(1)}_{j} h_{lj})\right) = \Big(\sum_{l=1}^{t} \pi^{(2)}_{1}(a_{l_{1}})\cdots \pi^{(2)}_{m}(a_{l_{m}})\Big) (\sum_{j=1}^{n} V^{(2)}_{j} h_{lj})
$$
for all $\Big( \sum\limits_{l=1}^{t} \pi^{(1)}_{1}(a_{l_{1}})\cdots \pi^{(1)}_{m}(a_{l_{m}})\Big) (\sum\limits_{j=1}^{n} V^{(1)}_{j} h_{lj}) \in {\rm span}\left\{\pi^{(1)}_{1}(\mathcal{A})\cdots \pi^{(1)}_{m}(\mathcal{A}) \big(\sum\limits_{j=1}^{n} V^{(1)}_{j}\mathcal{H}\big )\right\}.$
Now, we claim that U is an isometry. It is straight forward  to see that
\begin{align*}
&\left \Vert \Big(\sum_{l=1}^{t} \pi^{(2)}_{1}(a_{l_{1}})\cdots \pi^{(2)}_{m}(a_{l_{m}})\Big) (\sum_{j=1}^{n} V^{(2)}_{j} h_{lj}) \right\Vert^2 \\
 =&\left\{ \begin{array}{cc}
    \sum\limits_{l,r=1}^t\sum\limits_{i,j=1}^n  \big \langle \varphi_{ij}(a_{r_m}^*,a_{r_{m-1}}^*, \hdots,a_{r_1}^*a_{l_1}, \hdots , a_{l_{m-1}}, a_{l_m}) h_{lj} ,h_{ri} \big \rangle  &\text{if} \; k = 2m-1 \\
   &\\
    \sum\limits_{l,r=1}^t\sum\limits_{i,j=1}^n  \big \langle \varphi_{ij}(a_{r_m}^*,\hdots , a_{r_{m-1}}^*,\hdots , a_{r_1}^*, a_{l_1}, \hdots, a_{l_{m-1}} , a_{l_m}) h_{lj} , h_{ri} \big \rangle&\text{if} \; k = 2m. \\
\end{array} \right. 
\end{align*}
Since $\left(\{\pi^{(1)}_{j}\}_{j=1}^m, \{V^{(1)}_{j}\}_{j=1}^{n}, \mathcal{K}^{(1)}\right)$ is the minimal Stinespring triple of $\varphi,$ following the above procedure, we get
\begin{align*}
&\left \Vert \Big(\sum_{l=1}^{t} \pi^{(1)}_{1}(a_{l_{1}})\cdots \pi^{(1)}_{m}(a_{l_{m}})\Big) (\sum_{j=1}^{n} V^{(1)}_{j} h_{lj}) \right\Vert^2\\
 =&\left\{ \begin{array}{cc}
    \sum \limits_{l,r=1}^t\sum\limits_{i,j=1}^m  \big \langle \varphi_{ij}(a_{r_m}^*,a_{r_{m-1}}^*, \hdots,a_{r_1}^*a_{l_1}, \hdots , a_{l_{m-1}}, a_{l_m}) h_{lj} ,h_{ri} \big \rangle  &\text{if} \; k = 2m-1 \\
   &\\
    \sum\limits_{l,r=1}^t\sum\limits_{i,j=1}^m  \big \langle \varphi_{ij}(a_{r_m}^*,\hdots , a_{r_{m-1}}^*,\hdots , a_{r_1}^*, a_{l_1}, \hdots, a_{l_{m-1}}, a_{l_m}) h_{lj} ,h_{ri} \big \rangle&\text{if} \; k = 2m. \\
\end{array} \right. 
\end{align*}
Combining the last two equations, we conclude that 
$$\left \Vert U \Big(\sum_{l=1}^{t} \pi^{(1)}_{1}(a_{l_{1}})\cdots \pi^{(1)}_{m}(a_{l_{m}})\Big) (\sum_{j=1}^{n} V^{(1)}_{j} h_{lj}) \right\Vert =\left \Vert \Big(\sum_{l=1}^{t} \pi^{(1)}_{1}(a_{l_{1}})\cdots \pi^{(1)}_{m}(a_{l_{m}})\Big) (\sum_{j=1}^{n} V^{(1)}_{j} h_{lj}) \right\Vert.$$
Thus the map $U$ is  well defined and isometry. Therefore, $U$ can be extended uniquely from $\mathcal{K}^{(1)}$ to $\mathcal{K}^{(2)}$ preserving the isometry property. Without loss of generality, we denote this map by $U$. Clearly, the range of $U$ contains  ${\rm span}\left\{\pi^{(2)}_{1}(\mathcal{A})\cdots \pi^{(2)}_{m}(\mathcal{A}) 
\big(\sum\limits_{j=1}^{n} V^{(2)}_{j}\mathcal{H}\big )\right\}$ and it is dense in $\mathcal{K}^{(2)}.$ Therefore $U: \mathcal{K}^{(1)}\to \mathcal{K}^{(2)}$ is a unitary. Some easy computations yield that  $U$ satisfies the property (1) and (2).

\end{proof}
The following corollary is a multilinear version of Stinespring dilation theorem for block states.
\begin{cor}\label{Kaplan+ Multilinear} Let $\varphi= [\varphi_{ij}]\in ICP\big(M_{n}(\mathcal{A})^{k}, M_{n}(\mathbb{C})\big)$ be an invariant multilinear block CP map. Then there are $m$-commuting family of $*$-homomorphisms 
${\pi^{\varphi}_{i}:\A\to \mathcal{B(K)}},$  $1\leq i\leq m,$ and $\{f_{j}\}_{j=1}^{n} \subseteq \mathcal{K}$, 
for some Hilbert space $\mathcal{K}$ such that
\begin{enumerate}
    \item[(i)] if $k = 2m-1$, then $\varphi_{ij}$ is of the form:
    \begin{equation*}
        \varphi_{ij}(a_{1}, \hdots, a_{m-1}, a_{m}, a_{m+1}, \hdots, a_{2m-1}) = \langle \pi^{\varphi}_{1}(a_{m})\pi^{\varphi}_{2}(a_{m-1}a_{m+1}) \cdots \pi^{\varphi}_{m}(a_{1}a_{2m-1}) f_{j}, f_{i}\rangle,
    \end{equation*}
    for all $a_{1}, a_{2}, \hdots, a_{2m-1} \in \mathcal{A}$.
    \item[(ii)] if $k = 2m$, then $\varphi_{ij}$ is of the form:
    \begin{equation*}
        \varphi_{ij}(a_{1}, \hdots, a_{m-1}, a_{m}, a_{m+1}, \hdots, a_{2m}) = \langle\pi^{\varphi}_{1}(a_{m}a_{m+1})\pi^{\varphi}_{2}(a_{m-1}a_{m+2}) \cdots \pi^{\varphi}_{m}(a_{1}a_{2m})f_{j}, f_{i}\rangle
    \end{equation*}
    for all $a_{1}, a_{2}, \hdots, a_{2m} \in \mathcal{A}$.
   
\end{enumerate}
\end{cor}

\section{Radon Nikod\'ym Theorem for invariant block multilinear CP maps}\label{Rdaon}

In the theory of CP maps, the Radon Nikod\'ym derivative plays a crucial role in describing the relationship between two CP maps when one dominates the other \cite{W. Arveson69}.
Moreover, Radon-Nikod\'ym derivative provides a tool for characterizing the relative behavior of two quantum channels. It allows us to compare the effects of the channels on quantum states and analyze the extent to which one channel can be approximated by another. 
We explore the Radon Nikod\'ym Theorem for block multilinear CP maps, which is an application of the minimal Stinespring representation (refer to Theorem \ref{Theorem: Minimal BlockStinespring}).

To begin our discussion, it is necessary to establish an order relation ($\leq$) between two block multilinear CP maps.

 \begin{Def}   Let $\varphi= [\varphi_{ij}], \psi=[\psi_{ij}]\in ICP\big(M_{n}(\mathcal{A})^{k}, M_{n}(\mathbb{C})\big).$ 
 Then $\varphi$ dominates $\psi$ (in short $\psi \leq \varphi$) if 
  $$\varphi -\psi \in  ICP\big(M_{n}(\mathcal{A})^{k}, M_{n}(\mathbb{C})\big).$$
  \end{Def}
Let us assume that
 $\left(\{\pi^{(\varphi)}_{i}\}_{i=1}^m, \{V^{(\varphi)}_{i}\}_{i=1}^{n}, \mathcal{K}^{(\varphi)}\right)$ and 
 $\left(\{\pi^{(\psi)}_{i}\}_{i=1}^m, \{V^{(\psi)}_{i}\}_{i=1}^{n}, \mathcal{K}^{(\psi)}\right)$ be two minimal Stinespring triples of $\varphi$ and $\psi$ respectively. The commutant of the family 
 $\{\pi_{i}^{(\varphi)}\}_{i=1}^{m}$ is defined by the subspace of $\mathcal{B}(\mathcal{K}^{(\varphi)})$ by
 $$
\bigcap\limits_{i=1}^{m}\pi^{(\varphi)}_{i}(\mathcal{A})':=
 \{ T\in \mathcal{B}(\mathcal{K}^{(\varphi)}):  \pi_{i}^{(\varphi)}(a)T=
 T\pi_{i}^{(\varphi)}(a), \text{ for all } a\in \mathcal{A},  1\leq i \leq m\}.
  $$

\begin{thm}\label{thm: Radon Nikodym derivative} Let $\varphi= [\varphi_{ij}], \psi=[\psi_{ij}]\in ICP\big(M_{n}(\mathcal{A})^{k}, M_{n}(\mathbb{C})\big).$
Then $[\psi_{ij}]\leq  [\varphi_{ij}]$ if and only if there exists a $0\leq T \leq I$   with 
$$ T\in \bigcap\limits_{i=1}^{m}\pi_{i}(\mathcal{A})'$$
 such that each $\psi_{ij}$ can be expressed as follow:
   
  \begin{enumerate}
    \item[(i)] if $k = 2m-1$, then $\psi_{ij}$ is of the form:
 $$ \psi_{ij}(a_{1}, \hdots, a_{m-1}, a_{m}, a_{m+1}, \hdots, a_{2m-1}) = V_{i}^{(\varphi)\ast}T\pi^{(\varphi)}_{1}(a_{m})\pi^{(\varphi)}_{2}(a_{m-1}a_{m+1}) \cdots \pi^{(\varphi)}_{m}(a_{1}a_{2m-1})V^{(\varphi)}_{j},$$
 for all $a_{1}, a_{2}, \hdots, a_{2m-1} \in \mathcal{A}$.
    \item[(ii)] if $k = 2m$, then $\psi_{ij}$ is of the form:
    \begin{equation*}
        \psi_{ij}(a_{1}, \hdots, a_{m-1}, a_{m}, a_{m+1}, \hdots, a_{2m}) = V_{i}^{(\varphi)\ast}T\pi^{(\varphi)}_{1}(a_{m}a_{m+1})\pi^{(\varphi)}_{2}(a_{m-1}a_{m+2}) \cdots \pi^{(\varphi)}_{m}(a_{1}a_{2m})V^{(\varphi)}_{j},
    \end{equation*}
    for all $a_{1}, a_{2}, \hdots, a_{2m} \in \mathcal{A}$.
\end{enumerate}  
\end{thm}
\begin{proof}
It enough to prove only if part. Recall that $ [\psi_{ij}] \leq [\varphi_{ij}].$ Since 
$\left (\{\pi^{(\varphi)}_{i}\}_{i=1}^{m}, \{V^{(\varphi)}_{i}\}_{i=1}^{n}, \mathcal{K}^{(\varphi)}\right),$ and $\left(\{\pi^{(\psi)}_{i}\}_{i=1}^{m}, \{V^{(\psi)}_{i}\}_{i=1}^{n}, \mathcal{K}^{(\psi)}\right)$ are the minimal dilations of $\varphi$ and $\psi,$ therefore
\begin{align*}
    \mathcal{K}^{(\varphi)}&= \overline{{\rm span}}\left\{\pi^{(\varphi)}_{1}(\mathcal{A})\cdots \pi^{(\varphi)}_{m}(\mathcal{A}) \big(\sum_{i=1}^{n} V^{(\varphi)}_{i}\mathcal{H}\big )\right\},&
    \mathcal{K}^{(\psi)}= \overline{{\rm span}}\left\{\pi^{(\psi)}_{1}(\mathcal{A})\cdots \pi^{(\psi)}_{m}(\mathcal{A}) \big(\sum_{i=1}^{n} V^{(\psi)}_{i}\mathcal{H}\big )\right\}.
\end{align*}
In what follows,
we define a map 
$$S: {\rm span}\left\{\pi^{(\varphi)}_{1}(\mathcal{A})\cdots \pi^{(\varphi)}_{m}(\mathcal{A}) \big(\sum_{i=1}^{n} V^{(\varphi)}_{i}\mathcal{H}\big )\right\}
\to  {\rm span}\left\{\pi^{(\psi)}_{1}(\mathcal{A})\cdots \pi^{(\psi)}_{m}(\mathcal{A}) \big(\sum_{i=1}^{n} V^{\psi}_{i}\mathcal{H}\big )\right\}$$ by

$$
S\left ( \Big( \sum_{l=1}^{t} \pi^{(\varphi)}_{1}(a_{l_{1}})\cdots \pi^{(\varphi)}_{m}(a_{l_{m}})\Big) (\sum_{j=1}^{n} V^{(\varphi)}_{j} h_{lj})\right) = \Big(\sum_{l=1}^{t} \pi^{(\psi)}_{1}(a_{l_{1}})\cdots \pi^{(\psi)}_{m}(a_{l_{m}})\Big) (\sum_{j=1}^{n} V^{(\psi)}_{j} h_{lj})
$$
for all $\Big( \sum\limits_{l=1}^{t} \pi^{(\varphi)}_{1}(a_{l_{1}})\cdots \pi^{(\varphi)}_{m}(a_{l_{m}})\Big) (\sum\limits_{j=1}^{n} V^{(\varphi)}_{j} h_{lj}) \in {\rm span}\left\{\pi^{(\varphi)}_{1}(\mathcal{A})\cdots \pi^{(\varphi)}_{m}(\mathcal{A}) \big(\sum\limits_{j=1}^{n} V^{(\varphi)}_{j}\mathcal{H}\big )\right\}.$
To see, $S$ is well defined, observe that
\begin{align*}
  &\left \Vert  \Big(\sum_{l=1}^{t} \pi^{(\psi)}_{1}(a_{l_{1}})\cdots \pi^{(\psi)}_{m}(a_{l_{m}})\Big) (\sum_{j=1}^{n} V^{(\psi)}_{j} h_{lj}) \right\Vert^2 \\
  =&\left\{ \begin{array}{cc}
    \sum \limits_{l,r=1}^t\sum\limits_{i,j=1}^n  \big \langle \psi_{ij}(a_{r_m}^*,a_{r_{m-1}}^*, \hdots,a_{r_1}^*a_{l_1}, \hdots , a_{l_{m-1}}, a_{l_m}) h_{lj} ,h_{ri} \big \rangle  &\text{if} \; k = 2m-1 \\
   &\\
    \sum\limits_{l,r=1}^t\sum\limits_{i,j=1}^n  \big \langle \psi_{ij}(a_{r_m}^*,\hdots , a_{r_{m-1}}^*,\hdots , a_{r_1}^*, a_{l_1}, \hdots, a_{l_{m-1}}, a_{l_m}) h_{lj} ,h_{ri} \big \rangle&\text{if} \; k = 2m \\
\end{array} \right. \\
=&\left\{ \begin{array}{cc}
    \sum \limits_{l,r=1}^t  \big \langle [ \psi_{ij}(a_{r_m}^*,a_{r_{m-1}}^*, \hdots,a_{r_1}^*a_{l_1}, \hdots , a_{l_{m-1}}, a_{l_m})] 
    \begin{bmatrix}
    h_{l1}\\\vdots\\ h_{ln}
    \end{bmatrix}, \begin{bmatrix}
    h_{r1}\\\vdots\\ h_{rn}
    \end{bmatrix}
    \big \rangle  &\text{if} \; k = 2m-1 \\
   &\\
    \sum\limits_{l,r=1}^t  \big \langle[ \psi_{ij}(a_{r_m}^*,\hdots , a_{r_{m-1}}^*,\hdots , a_{r_1}^*, a_{l_1}, \hdots, a_{l_{m-1}}, a_{l_m})]\begin{bmatrix}
    h_{l1}\\\vdots\\ h_{ln}
    \end{bmatrix} ,\begin{bmatrix}
    h_{r1}\\\vdots\\ h_{rn}
    \end{bmatrix} \big \rangle&\text{if} \; k = 2m \\
\end{array} \right.\\
=&\left\{ \begin{array}{cc}
    \sum \limits_{l,r=1}^t  \left \langle \psi(D_{r_m}^*,D_{r_{m-1}}^*, \hdots,R_{r_1}^*R_{l_1}, \hdots , D_{l_{m-1}}, D_{l_m}) 
    h_{l},h_{r}
    \right \rangle  &\text{if} \; k = 2m-1 \\
   &\\
    \sum\limits_{l,r=1}^t  \left \langle  \psi (D_{r_m}^*,\hdots , D_{r_{m-1}}^*,\hdots , R_{r_1}^*, R_{l_1}, \hdots, D_{l_{m-1}}, D_{l_m}) h_{l} , h_{r} \right \rangle&\text{if} \; k = 2m, \\
\end{array} \right.
\end{align*}
 where $D_{r_p}, R_{l_{1}}$ are defined in Equation  \eqref{Eq: Essential}.
Similarly,
\begin{align*}
  &\left \Vert  \Big(\sum_{l=1}^{t} \pi^{(\psi)}_{1}(a_{l_{1}})\cdots \pi^{(\psi)}_{m}(a_{l_{m}})\Big) (\sum_{j=1}^{n} V^{(\psi)}_{j} h_{lj}) \right\Vert^2 \\
=&\left\{ \begin{array}{cc}
    \sum \limits_{l,r=1}^t   \left \langle \psi(D_{r_m}^*,D_{r_{m-1}}^*, \hdots,R_{r_1}^*R_{l_1}, \hdots , D_{l_{m-1}}, D_{l_m}) 
    h_{l},h_{r}
    \right \rangle  &\text{if} \; k = 2m-1 \\
   &\\
    \sum\limits_{l,r=1}^t  \left \langle  \psi (D_{r_m}^*,\hdots , D_{r_{m-1}}^*,\hdots , R_{r_1}^*, R_{l_1}, \hdots, D_{l_{m-1}}, D_{l_m}) h_{l} , h_{r} \right\rangle &\text{if} \; k = 2m, \\
\end{array} \right.
\end{align*}
 where $D_{r_p}, R{l_{1}}$ are defined in Equation  \eqref{Eq: Essential}.
 Since $\psi \leq \varphi,$ we have
 
    \begin{align*}\sum \limits_{l,r=1}^t  \big \langle \psi(D_{r_m}^*,D_{r_{m-1}}^*, \hdots,R_{r_1}^*R_{l_1}, \hdots , D_{l_{m-1}}, D_{l_m}) 
    h_{l},h_{r}  \big \rangle \\
    \leq \sum \limits_{l,r=1}^t 
    \big \langle \varphi(D_{r_m}^*,D_{r_{m-1}}^*, \hdots,R_{r_1}^*R_{l_1}, \hdots , D_{l_{m-1}}, D_{l_m}) 
    h_{l},h_{r}  \big \rangle
 \end{align*}
 whenever $k=2m-1,$ and
 
\begin{align*}\sum \limits_{l,r=1}^t \left \langle \psi(D_{r_m}^*,D_{r_{m-1}}^*, \hdots,R_{r_1}^*, R_{l_1}, \hdots , D_{l_{m-1}}, D_{l_m}) 
    h_{l},h_{r}  \right \rangle \\
    \leq \sum \limits_{l,r=1}^t 
    \left \langle \varphi(D_{r_m}^*,D_{r_{m-1}}^*, \hdots,R_{r_1}^*,R_{l_1}, \hdots , D_{l_{m-1}}, D_{l_m}) 
    h_{l},h_{r}  \right \rangle
 \end{align*}
  whenever $k=2m.$ Therefore, we have
$$
\left \Vert  \Big(\sum_{l=1}^{t} \pi^{(\psi)}_{1}(a_{l_{1}})\cdots \pi^{(\psi)}_{m}(a_{l_{m}})\Big) (\sum_{j=1}^{m} V^{(\psi)}_{j} h_{lj})\right\Vert^2 \leq 
\left \Vert  \Big(\sum_{l=1}^{t} \pi^{(\varphi)}_{1}(a_{l_{1}})\cdots \pi^{(\varphi)}_{m}(a_{l_{m}})\Big) (\sum_{j=1}^{m} V^{(\varphi)}_{j} h_{lj})\right\Vert^2.
$$
Consequently, $S: \mathcal{K}^{(\varphi)}\to  \mathcal{K}^{(\psi)}$ is well defined
and $\Vert S \Vert \leq 1$. Set $T= S^*S :\mathcal{K}^{(\varphi)}\to \mathcal{K}^{(\varphi)}$ so that $0\leq T\leq 1.$ It is easy to see
$$SV^{(\varphi)}_{j}= V_{j}^{(\psi)} \text{ and } S\pi^{(\varphi)}_{i}(a)=\pi^{(\psi)}_{i}(a)S$$ for all $a\in \mathcal{A}.$ Indeed,
\begin{align*}
    &S\pi^{(\varphi)}_{i}(a)\Big(\big(\sum_{l=1}^{t} \pi^{(\varphi)}_{1}(a_{l_{1}})\cdots \pi^{(\varphi)}_{m}(a_{l_{m}})\big) (\sum_{j=1}^{n} V^{(\varphi)}_{j} h_{lj})\Big)\\
    &=  S\Big(\pi^{(\varphi)}_{i}(a)\big(\sum_{l=1}^{t} \pi^{(\varphi)}_{1}(a_{l_{1}})\cdots \pi^{(\varphi)}_{m}(a_{l_{m}})\big) (\sum_{j=1}^{n} V^{(\varphi)}_{j} h_{lj})\Big)\\
    &= \pi^{(\psi)}_{i}(a)\sum_{l=1}^{t} \pi^{(\psi)}_{1}(a_{l_{1}})\cdots \pi^{(\psi)}_{m}(a_{l_{m}}) (\sum_{j=1}^{n} V^{(\psi)}_{j} h_{lj})\\
    &=\pi^{(\psi)}_{i}(a)S\Big( \sum_{l=1}^{t} \pi^{(\varphi)}_{1}(a_{l_{1}})\cdots \pi^{(\varphi)}_{m}(a_{l_{m}})\big) (\sum_{j=1}^{n} V^{(\varphi)}_{j} h_{lj}).\\
\end{align*}
Moreover, $S^*\pi^{(\psi)}_{i}(a)= \pi^{(\varphi)}_{i}(a)S^*,$ because, $S\pi^{(\varphi)}_{i}(a^*)=\pi^{(\psi)}_{i}(a^*)S.$ Since $S^*S=T,$ if follows from the above intertwiner  relations of $S$, we have
$$T\pi^{(\varphi)}_{i}(a)=\pi^{(\varphi)}_{i}(a)T$$ for all $a\in \mathcal{A}.$
Finally,  we notice that
\begin{align*}
    \psi_{ij}(a)
    =& \left\{ \begin{array}{cc}
    V_{i}^{(\psi)*} \pi_{1}^{(\psi)}(a_{1}a_{2m-1})
    \cdots \pi_{m}^{(\psi)}(a_{m})V_{j}^{(\psi)}  &\text{if} \; k = 2m-1 \\
   &\\
   V_{i}^{(\psi)*} \pi_{1}^{(\psi)}(a_{1}a_{2m-1})
    \cdots \pi_{m}^{(\psi)}(a_{m}a_{m+1})V_{j}^{(\psi)} &\text{if} \; k = 2m, \\
\end{array} \right.\\
    =&\left\{ \begin{array}{cc}
    (SV_{i}^{(\varphi)})^* \pi_{1}^{(\psi)}(a_{1}a_{2m-1})
    \cdots \pi_{m}^{(\psi)}(a_{m})(SV_{j}^{(\varphi)})  &\text{if} \; k = 2m-1 \\
   &\\
   (SV_{i}^{(\varphi)})^* \pi_{1}^{(\psi)}(a_{1}a_{2m-1})
    \cdots \pi_{m}^{(\psi)}(a_{m}a_{m+1})(SV_{j}^{(\varphi)}) &\text{if} \; k = 2m, \\
\end{array} \right.\\
 =&\left\{ \begin{array}{cc}
    V_{i}^{(\varphi)*}S^* \pi_{1}^{(\psi)}(a_{1}a_{2m-1})
    \cdots \pi_{m}^{(\psi)}(a_{m})SV_{j}^{(\varphi)}  &\text{if} \; k = 2m-1 \\
   &\\
    V_{i}^{(\varphi)*}S^* \pi_{1}^{(\psi)}(a_{1}a_{2m-1})
    \cdots \pi_{m}^{(\psi)}(a_{m}a_{m+1})SV_{j}^{(\varphi)} &\text{if} \; k = 2m, \\
\end{array} \right.\\
  =&\left\{ \begin{array}{cc}
    V_{i}^{(\varphi)*}S^*S \pi_{1}^{(\varphi)}(a_{1}a_{2m-1})
    \cdots \pi_{m}^{(\varphi)}(a_{m})V_{j}^{(\varphi)}  &\text{if} \; k = 2m-1 \\
   &\\
    V_{i}^{(\varphi)*}S^*S \pi_{1}^{(\varphi)}(a_{1}a_{2m-1})
    \cdots \pi_{m}^{(\varphi)}(a_{m}a_{m+1})V_{j}^{(\varphi)} &\text{if} \; k = 2m, \\
\end{array} \right.\\
=&\left\{ \begin{array}{cc}
    V_{i}^{(\varphi)*}T \pi_{1}^{(\varphi)}(a_{1}a_{2m-1})
    \cdots \pi_{m}^{(\varphi)}(a_{m})V_{j}^{(\varphi)}  &\text{if} \; k = 2m-1 \\
   &\\
    V_{i}^{(\varphi)*}T \pi_{1}^{(\varphi)}(a_{1}a_{2m-1})
    \cdots \pi_{m}^{(\varphi)}(a_{m}a_{m+1})V_{j}^{(\varphi)} &\text{if} \; k = 2m. \\
\end{array} \right.\\
\end{align*}
This completes the proof.
\end{proof}
We can see immediate consequence of the above Theorem as follows. 
\begin{rem}
By setting $k=1$ and $n=1$ in Theorem \ref{thm: Radon Nikodym derivative}, we can obtain the classical Radon-Nikodým Theorem for completely positive maps, which is stated in \cite{W. Arveson69}. Furthermore, as a direct consequence of Theorem~\ref{thm: Radon Nikodym derivative}, we can deduce Theorem 3.3 as presented in \cite{Heo}. For further details, please refer to \cite{Heo}.
\end{rem}
Let $\varphi = [\varphi_{ij}]\in ICP\left (M_{n}(\mathcal{A})^{k},  M_{n}(\mathcal{B(H)})\right).$ Then $\varphi$ is said to be pure if for any 
$\psi = [\psi_{ij}]\in ICP\left (M_{n}(\mathcal{A})^{k},  M_{n}(\mathcal{B(H)})\right)$ with $\varphi -\psi \in ICP\left (M_{n}(\mathcal{A})^{k},  M_{n}(\mathcal{B(H)})\right)$ ensures that 
$\psi = \lambda \varphi$ for some $\lambda \in [0,1].$

\begin{rem}    
In view of Theorem \ref{thm: Radon Nikodym derivative}, it is easy to see that a block multilinear  CP map
$\varphi$ is pure if and only if the associated von Neumann algebra $\bigcap\limits_{i=1}^{m}\pi^{(\varphi)}_{i}
(\mathcal{A})'$ is trivial.
\end{rem}

\section{Concluding remarks} \label{conclud remark}

In this section, we present some applications of Theorem \ref{Theorem: BlockStinespring}. By setting $n=1$ in Theorem \ref{Theorem: BlockStinespring}, we can recover more refined versions of \cite[Theorem 2.1]{Heo 00} and \cite[Theorem 2.1]{Heo} without requiring the symmetric and completely bounded assumptions, thanks to our previous results (see Proposition \ref{Prop: positivity implies symmetric} and Proposition \ref{invariant CP implies CB}). Moreover, by choosing $k=1$ in Theorem \ref{Theorem: BlockStinespring}, we can retrieve \cite[Corollary 2.4]{Heo99} by J. Heo. Similarly, by considering $k=1$ in Corollary \ref{Kaplan+ Multilinear}, we can obtain \cite[Theorem 2.1]{AK89} by A. Kaplan. Furthermore, Theorem \ref{thm: Radon Nikodym derivative} yields significant results when we set $k=1$. This choice allows us to recover the main result presented in \cite{Joita2}. Additionally, when considering the specific case of $n=1$, we are able to obtain \cite[Theorem 3.3]{Heo} as a direct consequence. Finally, we observe the Russo-Dye type theorem for invariant multilinear CP maps.
\begin{cor}\label{Cor: Multilinear Russo Dye}
Let $\varphi \in ICP (\mathcal{A}^k,\mathcal{B(H)}),~k \in \mathbb{N}$. Then $\|\varphi\|_{\mbox{cb}}=\|\varphi(1,1,\ldots,1)\|$.
\end{cor}

\begin{proof}
We shall provide an explicit argument for the case    $k=2m-1,~m \geq 1$. Similar procedure will follow for $k=2m,~m \geq 1$.
By the multilinear version of Stinespring Theorem, we get $\varphi(a_1,a_2,a_3,\ldots,a_{2m-1})=V^*\pi_{1}(a_m)\pi_{2}(a_{m-1}a_{m+1})\ldots \pi_{m}(a_1a_{2m-1})V$, for all $a_1,a_2,a_3,\ldots, a_{2m-1} \in \mathcal{A}$, where $V \in \mathcal{B(H, K)}$ and ${\pi_{i}:\A\to \mathcal{B(K)}},$ $i=1,2, \ldots, m$, are commuting $*$-homomorphisms (see \cite[Theorem 2.1]{Heo}). Let $\|a_i\| \leq 1$ for $i=1,2,3, \ldots, 2m-1$ and $h, k \in \mathcal{H}, \|h\| \leq 1, \|k\| \leq 1$. Now,
\begin{align*}
&\big \langle \varphi(a_1,a_2,a_3,\ldots,a_{2m-1})h,k  \big \rangle=\big \langle V^*\pi_{1}(a_m)\pi_{2}(a_{m-1}a_{m+1})\ldots \pi_{m}(a_1a_{2m-1})Vh,k \big \rangle\\
=&\big \langle \pi_{1}(a_m)\pi_{2}(a_{m-1}a_{m+1})\ldots \pi_{m}(a_1a_{2m-1})Vh,Vk \big \rangle.
\end{align*}
Putting both sides $\vert \cdot\vert$ we get, 
\begin{align*}
\big\vert \big \langle \varphi(a_1,a_2,a_3,\ldots,a_{2m-1})h,k  \big \rangle\big\vert=\big\vert \big \langle \pi_{1}(a_m)\pi_{2}(a_{m-1}a_{m+1})\ldots \pi_{m}(a_1a_{2m-1})Vh,Vk \big \rangle\big\vert \leq \|V\|^2.
\end{align*}
It gives us $\|\varphi \|\leq \|V\|^2$. But $\|\varphi(1,1,1, \ldots,1)\|=\|V^*V\|=\|V\|^2$, so we have $\|\varphi\| \leq \|\varphi(1,1,1, \ldots,1)\|$. Clearly $\|\varphi(1,1,1, \ldots,1)\| \leq \|\varphi\|$ and hence $\|\varphi\|= \|\varphi(1,1,1,\ldots,1)\|$.  Then,
\begin{align*}
&\varphi_2([a_{1_{ij}}],[a_{2_{ij}}],[a_{3_{ij}}],\ldots,[a_{{2m-1}_{ij}}])\\
=&\left[\sum_{r_1,r_2,\ldots,r_{2m-1}=1}^2 \varphi(a_{1_{ir_1}},a_{2_{r_1r_2}},a_{3_{r_2r_3}},\ldots,a_{{2m-1}_{r_{2m-2}j}})\right]\\
=&\left[\sum_{r_1,r_2,\ldots,r_{2m-1}=1}^2V^*\pi_{1}(a_{m_{r_{m-1}r_m}})\pi_{2}(a_{{m-1}_{r_{m-2}r_{m-1}}}a_{{m+1}_{r_mr_{m+1}}})\ldots \pi_{m}(a_{1_{ir_1}}a_{{2m-1}_{r_{2m-2j}}})V\right]\\
=&\begin{bmatrix} 
 V^*& &0\\
 0& &V^*
\end{bmatrix}
\left[\sum_{r_1,r_2,\ldots,r_{2m-1}=1}^2\pi_{m}(a_{1_{ir_1}})\ldots\pi_{1}(a_{m_{r_{m-1}r_m}})\ldots\pi_{m}(a_{{2m-1}_{r_{2m-2j}}})\right]
\begin{bmatrix} 
 V& &0\\
 0& &V
\end{bmatrix}\\
=&\begin{bmatrix} 
 V^*& &0\\
 0& &V^*
\end{bmatrix}
\left[\pi_{m}(a_{1_{ij}})\right]\ldots\left[\pi_{1}(a_{{m}_{ij}})\right]\ldots\left[\pi_{m}(a_{{2m-1}_{ij}})\right]
\begin{bmatrix} 
 V& &0\\
 0& &V
\end{bmatrix}\\
=&\begin{bmatrix} 
 V^*& &0\\
 0& &V^*
\end{bmatrix}
(\pi_{m})_2\big[a_{1_{ij}}\big]\ldots(\pi_{1})_2\big[a_{m_{ij}}\big]\ldots(\pi_{m})_2\big[a_{{2m-1}_{ij}}\big]
\begin{bmatrix} 
 V& &0\\
 0& &V
\end{bmatrix}.
\end{align*}
Now putting both side $\|.\|$ we get,
\begin{align*}
& \big\|\varphi_2\left(\big[a_{1_{ij}}\big], \big[a_{2_{ij}}\big],\big[a_{3_{ij}}\big],\ldots,\big[a_{{2m-1}_{ij}}\big]\right)\big\| 
\leq \big \|V^*\big\|\big\|V \big\| \big\|\big[a_{1_{ij}}\big]\big\|\big\|\big[a_{2_{ij}}\big] \big\| \big\|\big[a_{3_{ij}}\big]\big\|\ldots \big \|\big[a_{{2m-1}_{ij}}\big] \big\|.
\end{align*}
So, we have $\|\varphi_2\| \leq \|V\|^2$. 
Since $\|\varphi(1,1,1, \ldots,1)\|=\|V\|^2$,  we get $\|\varphi_2\|\leq \|\varphi(1,1,1,\ldots,1)\|=\|\varphi\| \leq \|\varphi_2\|$. Similarly we can show that $\|\varphi_n\|=\|\varphi(1,1,1,\ldots,1)\|$ for all $n \in \mathbb{N}$. This gives us $\|\varphi\|_{\mbox{cb}}=\|\varphi(1,1,1,\ldots,1)\|$.
This completes the proof.
\end{proof}

\vspace{0.2in}

\noindent\textit{Acknowledgements:}
The research of first named author is supported by the Theoretical Statistics and Mathematics Unit, Indian Statistical Institute, Bangalore Centre, India and J. C. Bose Fellowship of SERB (India) of Prof. B V Rajarama Bhat. 
The research of the second named author is supported by the NBHM postdoctoral fellowship, Department of Atomic Energy (DAE), Government of India (File No: 0204/3/2020/R$\&$D-II/2445). Both authors would like to express their deep appreciation to Prof. Mathew Graydon from Waterloo University, specifically the QIT groups, for sharing the profound connection between multilinear block CP maps and quantum information theory.
The authors would like to thank Professor Erik Christensen for fruitful discussions. Both authors would like to thank Professor Michael Skeide for his insightful comments to improve the readability of the manuscript.
The authors would like to express their gratitude to Prof. B V Rajarama Bhat and Prof. Jaydeb Sarkar for their invaluable guidance and inspiration throughout the course of this project.


\end{document}